\documentclass[11pt]{amsart}
\usepackage{amscd,amsmath,amssymb,amsfonts,verbatim}
\usepackage[cmtip, all]{xy}
\usepackage[letterpaper, left=3cm, right=3cm, top = 2.5cm, bottom = 2.5cm ]{geometry}
\usepackage{comment}
\usepackage{MnSymbol}

\newtheorem{thm}{Theorem}[section]
\newtheorem{prop}[thm]{Proposition}
\newtheorem{lem}[thm]{Lemma}
\newtheorem{cor}[thm]{Corollary}

\theoremstyle{definition}

\newtheorem{defn}[thm]{Definition}
\theoremstyle{remark}
\newtheorem{remk}[thm]{Remark}
\newtheorem{remks}[thm]{Remarks}

\newtheorem{exm}[thm]{Example}
\newtheorem{exms}[thm]{Examples}
\newtheorem{notat}[thm]{Notation}
\numberwithin{equation}{section}

{\hfill$\square$\end{defn}}
{\hfill$\square$\end{remk}}
{\hfill$\square$\end{remks}}
{\hfill$\square$\end{exm}}
{\hfill$\square$\end{exms}}
{\hfill$\square$\end{notat}}

\newcommand{\thmref}{Theorem~\ref}
\newcommand{\propref}{Proposition~\ref}
\newcommand{\corref}{Corollary~\ref}
\newcommand{\defref}{Definition~\ref}
\newcommand{\lemref}{Lemma~\ref}

\newcommand{\sF}{{\mathcal F}}

\newcommand{\sH}{{\mathcal H}}

\newcommand{\sK}{{\mathcal K}}
\newcommand{\sL}{{\mathcal L}}
\newcommand{\sM}{{\mathcal M}}

\newcommand{\sO}{{\mathcal O}}
\newcommand{\sP}{{\mathcal P}}

\newcommand{\sR}{{\mathcal R}}

\newcommand{\sZ}{{\mathcal Z}}
\newcommand{\A}{{\mathbb A}}

\newcommand{\C}{{\mathbb C}}
\newcommand{\D}{{\mathbb D}}

\newcommand{\G}{{\mathbb G}}
\renewcommand{\H}{{\mathbb H}}

\renewcommand{\P}{{\mathbb P}}
\newcommand{\Q}{{\mathbb Q}}

\newcommand{\Z}{{\mathbb Z}}

\newcommand{\fm}{{\mathfrak m}}
\newcommand{\fn}{{\mathfrak n}}
\newcommand{\fp}{{\mathfrak p}}
\newcommand{\fq}{{\mathfrak q}}

\newcommand{\Alb}{{\rm Alb}}
\newcommand{\CH}{{\rm CH}}

\newcommand{\surj}{\twoheadrightarrow}
\newcommand{\inj}{\hookrightarrow}
\newcommand{\red}{{\rm red}}

\newcommand{\Pic}{{\rm Pic}}

\newcommand{\Spec}{{\rm Spec \,}}

\newcommand{\divf}{{\rm div}}

\newcommand{\id}{{\operatorname{id}}}

\newcommand{\Sch}{{\operatorname{\mathbf{Sch}}}}

\newcommand{\<}{\langle}
\renewcommand{\>}{\rangle}

\newcommand{\Sm}{{\mathbf{Sm}}}

\newcommand{\can}{{\operatorname{\rm can}}}

\newcommand{\cyc}{{\operatorname{\rm cyc}}}

\newcommand{\et}{{\text{\'et}}}

\newcommand{\ds}{{/\kern-3pt/}}

\newcommand{\un}{\underline}
\newcommand{\ov}{\overline}

\newcommand{\dgn}{{\operatorname{degn}}}
\renewcommand{\dim}{\text{\rm dim}}

\newcommand{\tuborg}{\left\{\begin{array}{ll}}
\newcommand{\sluttuborg}{\end{array}\right.}

\newcommand{\zar}{{\rm zar}}
\newcommand{\nis}{{\rm nis}}
\newcommand{\wt}{\widetilde}

\newcommand{\prolim}{\textnormal{ `` $ \underset{n}{\varprojlim}$" } }

\newcounter{elno}

\newcounter{elno-abc}   

\newcounter{elno-abc-prime}

\begin{document}

\title{K-theory and 0-cycles on schemes}
\author{Rahul Gupta, Amalendu Krishna}
\address{School of Mathematics, Tata Institute of Fundamental Research,  
1 Homi Bhabha Road, Colaba, Mumbai, India}
\email{amal@math.tifr.res.in}
\email{rahul@math.tifr.res.in}


\keywords{Algebraic $K$-theory, algebraic cycles, singular varieties}        

\subjclass[2010]{Primary 14C25; Secondary 19E08, 19E15}

\maketitle

\begin{quote}\emph{Abstract.}  We prove Bloch's formula for 0-cycles on
affine schemes over algebraically closed fields.
We prove this formula also for projective schemes over algebraically closed 
fields which are regular in codimension one. 
Several applications, including 
Bloch's formula for 0-cycles with modulus, are derived.
\end{quote}
\setcounter{tocdepth}{1}
\tableofcontents

\section{Introduction}\label{sec:Intro}
The principal aim of this paper is to study the Chow group of 0-cycles on
singular schemes and the Chow group of 0-cycles with modulus on smooth
schemes. We prove Bloch's formula for these groups and show that the 
canonical cycle class map from them to the appropriate $K$-theory and relative
$K$-theory groups have torsion kernels of bounded exponents. 
These results directly extend the analogous classical results about the
0-cycle groups on smooth schemes to the setting of 0-cycles on singular schemes
and 0-cycles with modulus on smooth
schemes. We derive several outstanding consequences of these results.
This section provides the background of these problems, a summary of
main results and applications, their statements and outline of proofs.

\subsection{Bloch's formula}\label{sec:Background}
The Bloch-Quillen formula in the theory of
algebraic cycles provides a description for the Chow group of 0-cycles on a 
smooth quasi-projective scheme over a field in terms of the Zariski
cohomology of the Quillen $K$-theory sheaves. 
This yields a direct connection
between the Chow groups and algebraic $K$-theory of smooth quasi-projective
schemes. This formula for curves is classical and
the case of surfaces was derived by Bloch \cite{Bloch-0}. The general
case of the formula for all smooth quasi-projective schemes was established
by Quillen \cite{Quillen}.  Kato \cite{Kato} showed that
this formula also holds if one replaces the Quillen $K$-theory 
sheaves by the Milnor $K$-theory sheaves and the Zariski cohomology by
the Nisnevich cohomology. In conclusion, for a smooth 
quasi-projective scheme $X$ of dimension $d \ge 0$ over a field,
one knows that
\begin{equation}\label{eqn:sm}
\CH^d(X) \simeq  H^d_{\zar}(X, \sK^M_{d,X}) \simeq H^d_{\nis}(X, \sK^M_{d,X})
\simeq H^d_{\zar}(X, \sK_{d,X}).
\end{equation}

However, no complete generalization of this formula has been found for
singular schemes despite the fact that there is a 
well established theory of 0-cycles on such schemes after the work of 
Levine and Weibel \cite{LW} in 1980's. 
This formula was shown for singular curves by Levine and Weibel \cite{LW}.
Collino \cite{Collino} proved the Bloch-Quillen formula for a scheme
which is almost non-singular (meaning that it has only one singular point).
For a quasi-projective surface whose singular locus is affine, 
a Bloch-Quillen formula of the type ~\eqref{eqn:sm}
was proven by Pedrini and Weibel \cite{PW}. 
Levine \cite{Levine-1} proved this formula for all singular surfaces over
algebraically closed fields (see \cite{BS-*} for details). 
These are the only cases of singular schemes for which any of the isomorphisms 
in ~\eqref{eqn:sm} is presently known.

On the contrary, Levine and Srinivas showed \cite[\S~3.2]{Srinivas-1} that the 
formula of Quillen for the Chow group of 0-cycles can not be generalized
to singular schemes, even with the rational coefficients.
They showed that if $X$ is the boundary of the 4-simplex in $\A^4_{\C}$, given
by the equation $f(x, y, z,w) = xyzw(1-x-y-z-w) = 0$, then
$\CH^3(X)_{\Q} \cong K^M_3(\C)_{\Q}$, whereas $H^3_{\zar}(X, \sK_{3,X})_{\Q}
\cong K_3(\C)_{\Q}$. 

\vskip .2cm

This example shows that Quillen's generalization of Bloch's formula 
can not be extended to higher dimensional affine schemes.
However, it is an open question at present if Kato's generalization of Bloch's formula extends to singular schemes.
Our first main result answers this question as follows.
For any Noetherian scheme $X$, let us denote its Nisnevich (resp. Zariski)
site by $X_{\nis}$ (resp. $X_{\zar}$). For a separated and reduced
scheme $X$ of finite type over field, we let $\CH^{LW}_0(X)$ denote
the Levine-Weibel Chow group of 0-cycles.   
Let $\sK^M_{i, X}$ denote the Nisnevich (or Zariski) sheaves of Milnor
$K$-groups on $X$ (see \S~\ref{sec:Prelim}).

\begin{thm}\label{thm:Thm-1}
Let $k$ be an infinite perfect field and let $X$ be a reduced quasi-projective  
scheme of pure dimension $d \ge 0$ over $k$. 
Then there is a canonical surjective map
\begin{equation}\label{eqn:Main-0}
\rho_X: \CH^{LW}_0(X) \surj H^d_{\nis}(X, \sK^M_{d,X}).
\end{equation}

This is an isomorphism if $k$ is algebraically closed and either
of the following holds.
\begin{enumerate} 
\item 
$X$ is affine.
\item
$X$ is projective and regular in codimension one.
\end{enumerate}
\end{thm}

In our proof of the isomorphism, 
the assumption that $k$ is algebraically closed plays
a crucial role. Without this assumption, we can prove
\thmref{thm:Thm-1} for affine surfaces and for a modified version of
the Chow group (see \S~\ref{sec:Chow-sing}).

\begin{thm}\label{thm:Thm-2}
Let $X$ be a reduced affine surface over a perfect field $k$. 
Then there is a canonical isomorphism
\[
\rho_X: \CH_0(X) \xrightarrow{\cong} H^2_{\nis}(X, \sK^M_{2,X}).
\]
\end{thm}

\vskip .2cm

\subsection{Bloch's formula for Chow group with 
modulus}\label{sec:modulus}
Just as Bloch's higher Chow groups are the motivic cohomology which
describe algebraic $K$-theory of smooth schemes, the higher Chow groups
with modulus \cite{BSaito} are supposed to describe the relative
algebraic $K$-theory of the pair $(X,D)$, where $X$ is a 
smooth scheme and $D \subset X$ is an effective Cartier divisor. 
One of the primary goals of the program of connecting cycles with modulus
and relative $K$-theory is the proof of
Bloch-Quillen-Kato type formula for the Chow groups with modulus. 
As an application of 
Theorems~\ref{thm:Thm-1} and ~\ref{thm:Thm-2}, we solve this problem
for 0-cycles with modulus as follows.

\begin{thm}\label{thm:Thm-3}
Let $k$ be an algebraically closed field and let 
$X$ be a smooth quasi-projective 
scheme of dimension $d \ge 1$ over $k$. Let $D \subset X$ be an effective
Cartier divisor. Then there is a canonical isomorphism
\begin{equation}\label{eqn:Main-1}
\rho_{X|D}: \CH_0(X|D) \xrightarrow{\cong} H^d_{\nis}(X, \sK^M_{d,(X,D)})
\end{equation}
in the following cases.
\begin{enumerate} 
\item 
$X$ is affine.
\item
$X$ is projective and $D$ is integral.
\end{enumerate}
\end{thm}

If $d \le 2$, then part (2) of the above theorem holds without any condition
on $D$ and this was shown in \cite{BK}.
If $k$ is not necessarily algebraically closed, we can prove the following.

\begin{thm}\label{thm:Thm-4}
Let $X$ be a smooth affine surface over a perfect field $k$ and
let $D \subset X$ be an effective Cartier divisor. 
Then there is a canonical isomorphism
\[
\rho_{X|D}: \CH_0(X|D) \xrightarrow{\cong} H^2_{\nis}(X, \sK^M_{2,(X,D)}).
\]
\end{thm}

\vskip .2cm

\subsection{The cycle class map}\label{sec:CCM*}
For a Noetherian scheme $X$, 
there is a cycle class map $\lambda_X: \CH^{LW}_0(X) \to K_0(X)$
(see \cite[Proposition~2.1]{LW}). Let $F^dK_0(X)$ be its image.
If $X$ is a smooth scheme and $D \subset X$ is an effective Cartier
divisor, there is a cycle class map $\lambda_{X|D}: \CH_0(X|D) \to K_0(X,D)$
(see \cite{BK} and \S~\ref{sec:CCM-M}).
As an application of the Chern class maps,  
Grothendieck \cite[\S~4.3]{Grothendieck} proved that
for a smooth quasi-projective scheme $X$ of dimension $d\geq 1$ over a field, 
the kernel of the cycle class 
map $\lambda_X$ is a torsion group of exponent $(d-1)!$.

The second principal aim of this text is to generalize this result to
0-cycles on singular schemes and 0-cycles with modulus on smooth schemes
as follows. When $k = \ov{k}$, part (1) of the following theorem
gives an independent proof of
an old unpublished result of Levine (see \cite[Corollary~5.4]{Levine-5}) 
for affine schemes. 
When $k$ is not algebraically closed,this result is completely new.
Recall that an affine algebra over $k$ means a finite type $k$-algebra.

\begin{thm}\label{thm:Thm-6}
Let $A$ be a geometrically reduced affine algebra of dimension $d \ge 1$ over a field $k$. Assume that either $k$ is algebraically
closed or $(d-1)! \in k^{\times}$. Let $X = \Spec(A)$ and let $D \subset X$ be an
effective Cartier divisor. Then the following hold.  
\begin{enumerate}

\item If $k$ is an infinite field, then the kernel of the cycle class
map $\lambda_A:\CH^{LW}_0(A) \to K_0(A)$ is a torsion group of exponent
$(d-1)!$.
\item 
For arbitrary fields, the kernel of the cycle class
map $\lambda_A:\CH_0(A) \to K_0(A)$ is a torsion group of exponent
$(d-1)!$, where as before $\CH_0(A)$ is a modified version of the Chow group. 
\item If $k$ is perfect and $X$ is smooth, then the kernel of the  cycle class map 
$\lambda_{X|D}: \CH_0(X|D) \to K_0(X,D)$ 
is a torsion group of exponent
$(d-1)!$.
\end{enumerate}
\end{thm} 

\vskip .3cm
 
We now give several applications of the above results.

\subsection{Motivic cohomology with modulus}

Let $X$ be a smooth quasi-projective scheme of pure dimension $d$ over a field $k$ and let $D \subset X$ be an effective Cartier divisor. Let $H^i_{\sM}(X|D, \Z(r))$ denote the motivic cohomology of the pair $(X,D)$.
This is defined as the Nisnevich hypercohomology of the presheaf of
cycle complexes with modulus. 
There exists a canonical map $\CH^r(X|D, 2r-i) \to H^i_{\sM}(X|D, \Z(r))$
(see \S~\ref{sec:MCM}). 
It is known that this natural map is an isomorphism if $D= \emptyset$. 
For an arbitrary effective divisor $D$, one does not know much about the
nature of this map. As an application of Bloch's formula for 0-cycles with 
modulus, we show the following.

\begin{thm} \label{thm:CM-MCM-1}
Let $k$ be an algebraically closed field.  Let $X$ be a smooth quasi-projective scheme of pure dimension $d\geq 1$ over $k$ and let $D \subset X$ be an effective Cartier divisor such that $D_{{\rm red}}$ is a simple normal crossing divisor. 
 Then the map 
$\can_{\sM}: \CH^d(X|D) \to H^{2d}_{\sM}(X|D, \Z(d))$ of ~\eqref{eqn:CM-MCM} is surjective in the following cases. 
\begin{enumerate}
\item $d\leq 2$. 
\item $X$ is affine.
\item ${\rm char}(k) = 0$ and $X$ is projective.  
\item ${\rm char}(k) > 0$,  $X$ is projective and $D$ is reduced.
\end{enumerate}
Moreover, it is an isomorphism if in addition, $D$ is connected and smooth. 
\end{thm}

If $k$ is a finite field and $D_{{\rm red}}$ is a simple normal crossing divisor, then R{\"u}lling and Saito \cite{RS} proved that the map $\can_{\sM}$ induces an isomorphism of the pro-abelian groups  
$\can_{\sM} :\textnormal{ `` $ \underset{n}{\varprojlim}$" } 
\CH^d(X|nD) \to \textnormal{ `` $ \underset{n}{\varprojlim}$" } 
H^{2d}_{\sM}(X|nD, \Z(d))$.
Using the results of R{\"u}lling and Saito, and \thmref{thm:Thm-3} as new ingredient, 
we can prove this isomorphism for affine schemes
and for quasi-projective surfaces over algebraically closed fields. More precisely, we have

\begin{thm} \label{thm:CM-MCM-2}
Let $k$ be an algebraically closed field. Let $X$ be a smooth quasi-projective scheme of pure dimension $d\geq 1$ over $k$ and let $D \subset X$ be a simple normal crossing divisor.
Then the map of pro-abelian groups $\can_{\sM} :\prolim \CH^d(X|n D) \to \prolim  H^{2d}_{\sM}(X|n D, \Z(d))$
is an isomorphism in the following cases.
\begin{enumerate}
\item $X$ is affine. 
\item $X$ is a quasi-projective scheme and $d=2$. 
\end{enumerate}
\end{thm}

\subsection{The strong Bloch-Srinivas conjecture}\label{sec:BSC}
Let $X$ be a reduced affine or projective scheme of dimension $d \ge 1$ over an
algebraically closed field $k$. Assume that $X$ is regular in codimension
one and there exists a resolution of singularities $\pi: \wt{X} \to X$.
Let $E_0 \subset \wt{X}$ be the reduced exceptional divisor. It is known
that there exists a surjective 
pull-back map $\pi^*: \CH^{LW}_0(X) \to \CH_0(\wt{X})$.
It is not hard to see that this map has a factorization
$\CH^{LW}_0(X) \xrightarrow{\pi^*_n} \CH_0(\wt{X}|nE_0) \surj \CH_0(\wt{X})$ 
for every $n \ge 1$.
As an application of our proof of \thmref{thm:Thm-1}, we can prove the following
result about the maps $\pi^*_n$.

\begin{thm}\label{thm:Thm-5}
The map $\pi^*_n: \CH^{LW}_0(X) \to \CH_0(\wt{X}|nE_0)$
is an isomorphism for all $n \gg 0$.
If ${\rm char}(k) > 0$, then $\pi^*_n$ is an isomorphism for every $n \ge 1$.
\end{thm}

The first part of Theorem~\ref{thm:Thm-5} was conjectured  
(in a different but equivalent form) by 
Bloch and Srinivas  \cite{Srinivas-2} for normal surfaces. Its proof was given
in \cite{KS}. This conjecture allows us to estimate
the kernel of the map $\CH^{LW}_0(X) \surj \CH_0(\wt{X})$. 

\subsection{Question of Kerz and Saito}\label{sec:KSQ}
Let $X$ be a smooth projective scheme of dimension $d \ge 1$ over a perfect
field $k$ of positive characteristic 
and let $U \subset X$ be an open subset 
whose complement is supported on a divisor. Then a question posed by
Kerz and Saito \cite[Question~V]{KeS} asks if there is an
isomorphism
\[
{\underset{D}\varprojlim} \ \CH_0(X|D) \xrightarrow{\cong} 
{\underset{D}\varprojlim} \ H^d_{\nis}(X, \sK^M_{d,(X,D)}),
\]
where the limits are taken over effective divisors on $X$ with
support outside $U$.

This question was answered positively by Kerz and Saito if $k$ is a
finite field, using an earlier result of Kato and Saito \cite{Kato-Saito}.
As an application of \thmref{thm:Thm-5}, we prove the following stronger
version of this question whenever $k$ is algebraically closed and
$X \setminus U$ can be contracted
to a smaller dimensional scheme without changing $U$.

\begin{thm}\label{thm:Thm-7}
Let $Y$ be a reduced projective scheme of pure dimension $d \ge 1$ over an
algebraically closed field $k$ of positive characteristic. 
Assume that $Y$ is regular in codimension
one and there exists a resolution of singularities $\pi: X \to Y$.
Let $E_0 \subset X$ be the reduced exceptional divisor. Then for any
effective divisor $D \subset X$ with support $E_0$, there is a commutative
diagram
\begin{equation}\label{eqn:Thm-7-0}
\xymatrix@C.8pc{
\CH_0(X|D) \ar[r]^-{\rho_{X|D}} \ar[d] & H^d_{\nis}(X, \sK^M_{d,(X,D)}) 
\ar[d] \\
\CH_0(X|E_0) \ar[r]^-{\rho_{X|E_0}}  & H^d_{\nis}(X, \sK^M_{d,(X,E_0)})}
\end{equation}
in which all arrows are isomorphisms.
\end{thm}

We warn the reader that this result does not follow from \thmref{thm:Thm-3}.
Rather, it gives new cases of Bloch's formula for Chow groups with modulus.

\subsection{Schlichting's theorem}\label{sec:Sch}
In \cite{Sch}, Schlichting gave a necessary and sufficient condition for 
vector bundles of top rank on affine schemes to admit nowhere vanishing 
sections. By combining \thmref{thm:Thm-1} with \cite[Theorem~1.2]{Krishna-2},
we recover Schlichting's theorem over algebraically closed fields.

\begin{cor}\label{cor:Sch-thm}
Let $A$ be a reduced affine algebra of pure dimension $d \ge 1$ 
over an algebraically closed field and let $X = \Spec(A)$.
Let $P$ be a projective $A$-module of rank $d$. Then $P$ admits an
Euler class $e(P) \in H^d_{\zar}(X, \sK^M_{d,X})$.
Furthermore, $P$  splits off 
a free summand of positive rank if and only if $e(P)$ dies in
$H^d_{\nis}(X, \sK^M_{d,X})$.
\end{cor}

Indeed, \cite[Theorem~1.2]{Krishna-2} says that $P$  splits off 
a free summand of positive rank if and only if $c_d(P)=0$ in $K_0(X)$.
On the other hand, $c_d(P)$ lies in  $F^d K_0(X)$ by \cite[Remark~3.6]{Murthy}
and $\CH_0^{LW}(X) \xrightarrow{\cong} F^d K_0(X)$ by \cite[Theorem~1.3]{Krishna-2}. 
\corref{cor:Sch-thm} therefore follows from  \thmref{thm:Thm-1}.

\subsection{Euler class group and Chow group}\label{sec:ECCGrp}
The Euler and weak Euler class groups
of a commutative Noetherian ring $A$ were introduced by
Bhatwadekar and R. Sridharan \cite{BS-3} in order to study the question of
existence of nowhere vanishing sections of projective modules of rank
$= \dim(A)$. If $A$ is a smooth affine algebra over an infinite perfect
field, it was conjectured by Bhatwadekar and R. Sridharan 
(see \cite[Remark~3.13]{BS-4}) that the weak Euler class group $E_0(A)$
coincides with the Chow group of 0-cycles $\CH_0(A)$.
This was proven by Bhatwadekar (unpublished) if $\dim(A) \le 2$ and
by Asok and Fasel \cite{AF} in general. As part of the proof of 
\thmref{thm:Thm-2}, we establish the following partial generalization of 
these results to non-smooth algebras.

\begin{thm}\label{thm:ECG-*}
Let $A$ be a 2-dimensional geometrically reduced affine algebra over an infinite
field. Then there is an isomorphism 
\[
E_0(A) \xrightarrow{\cong} \CH^{LW}_0(A).
\]
\end{thm}

\vskip .3cm

\subsection{Outline of proofs}\label{sec:outline}
We now briefly outline the strategy of our proofs.
Bloch's formula for 0-cycles on singular schemes has three main
steps: the construction of the Bloch-Quillen map,
showing its surjectivity and, its injectivity. 

The construction of the Bloch-Quillen map from
the Chow group to the cohomology of the Milnor $K$-theory sheaves 
is a major obstacle. In the smooth case, this directly
follows from the Gersten complex. But this method breaks down in the singular 
case. Another possibility is to construct this map for surfaces and
then reduce the general case to the surface case. But this too breaks down
due to the lack of a push-forward map on the top cohomology of the Milnor 
$K$-theory sheaf.

Instead, what we observe here is that in any dimension, the Cousin complex still
gives a presentation of the desired cohomology group.
The heart of the proof then is to compare some part of the
Gersten complex with Cousin complex of Milnor $K$-theory
sheaves to kill the rational equivalence on the 
group of 0-cycles. The key role here is played by our
Proposition~\ref{prop:Key}.

We complete the proof of
\thmref{thm:Thm-1} in \S~\ref{sec:Thm-1-prf}. The main ingredients for
the surjectivity 
are some results of Kato and Saito \cite{Kato-Saito}.
The injectivity is shown using the Roitman
torsion theorems of \cite{Krishna-2} and \cite{KS}. Apart from these, 
we also need to use a technique of Levine \cite{Levine-2} to study the
relation between the $K$-theory of a normal projective scheme and 
its albanese variety.

In \S~\ref{sec:BF-mod}, we prove Bloch's formula for the Chow group with
modulus. In order to do so, we generalize \thmref{thm:Thm-1}
to certain kind of projective schemes which are not regular in codimension
one. The crucial ingredient here is the Roitman torsion theorem of
\cite{Krishna-3}. We derive Bloch's formula in the modulus setting
using this and a decomposition theorem for the Chow group of 0-cycles
from \cite{BK}. In this section, 
we also prove Theorems~\ref{thm:CM-MCM-1} and ~\ref{thm:CM-MCM-2}. 
The main idea here is to use a theorem of 
R{\"u}lling and Saito \cite{RS} which compares motivic cohomology with modulus 
with the cohomology of a relative Milnor $K$-group. 
The theorems then follow from our Bloch's formula with modulus 
(\thmref{thm:Thm-3}) by comparing the relative Milnor $K$-groups of 
R{\"u}lling-Saito  and Kato-Saito.

The strong Bloch-Srinivas conjecture is proven in \S~\ref{sec:BSC-prf}
using \thmref{thm:Thm-1}, the recent pro-descent theorem of 
Kerz, Strunk and Tamme \cite{KST} and some results on the $K$-theory
in positive characteristic from \cite{Krishna-2}. A question of 
Kerz-Saito is answered in a special case
as an application of our proof of the strong 
version of the Bloch-Srinivas conjecture.

In \S~\ref{sec:ECG}, we prove \thmref{thm:Thm-6}.
The proof is delicate and goes into several steps. 
We use the theory of Euler class groups of commutative Noetherian
rings. For singular rings, these groups are difficult to study directly.
To circumvent this, we introduce a modified version of the
Euler class group.
We then show that this modified version coincides with the
classical one for singular affine schemes. It uses a hard
result of Van der Kallen \cite{vdK-1}, a theorem of
Das-Zinna \cite{DZ} and the Bertini theorem of Murthy and Swan.
 
Once this isomorphism is achieved, we use 
the Bertini theorems of Murthy and Swan (once again),
the cancellation theorem of Suslin and some results of 
Bhatwadekar-R. Sridharan on the Euler classes of projective modules to 
complete the proof of \thmref{thm:Thm-6}.
Using these Euler class groups, we derive 
Bloch's formula for affine surfaces over arbitrary perfect fields in 
\S~\ref{sec:gen-field}.

\subsection{Notations} 
The following notations will be followed in this text.
The word {\sl scheme} will mean a separated Noetherian
scheme of finite Krull dimension and the word {\sl ring} will mean a 
commutative Noetherian ring.
For a scheme $X$, the normalization of $X_{\rm red}$ will be denoted by $X^N$.
We shall denote the
Nisnevich (resp. Zariski) site of $X$ by $X_{\nis}$ (resp. $X_{\zar}$).
For a point $x \in X$, we shall denote the scheme $\Spec(\sO_{X,x})$
by $X_x$. We let $X^o_x = X_x \setminus \{x\}$ and $\eta_x = \Spec(k(x))$.
For a closed subscheme $Z \subset X$, we shall let $|Z|$ denote the 
support of $Z$.

Throughout this text, we
shall fix a perfect field $k$ and let $\Sch_k$ denote the 
category of separated schemes of finite type over $k$. 
We shall let $\Sm_k$ denote the category
of those schemes in $\Sch_k$ which are smooth over $k$. 
For $X, Y \in \Sch_k$, we shall denote $X \times_{\Spec(k)} Y$ simply
by $X \times Y$.

For abelian groups $A$ and $B$, we shall write $A \otimes_{\Z} B$ in short as
$A \otimes B$. For a prime $p$, we shall let 
$A\{p\}$ denote the $p$-primary torsion subgroup of $A$.

\section{Review of 0-cycles  and Milnor 
$K$-sheaves}\label{sec:Prelim}
In this section, we recall the definitions of various 0-cycle groups
and relations between them. We also recall the definition of the
Milnor $K$-theory sheaves which is one of our main objects of study
in order to prove the Bloch-Quillen type formula for Chow groups 
on singular varieties and Chow groups with modulus.
We also prove some other preliminary results that will be used in the
proofs of the main results.

\subsection{Levine-Weibel Chow group of singular schemes}
\label{sec:Chow-sing}
We recall the definition of the cohomological Chow group of 0-cycles for
singular schemes from \cite{BK} and \cite{LW}. 
Let $X$ be a reduced quasi-projective
scheme of dimension $d \ge 1$ over $k$. Let $X_{\rm sing}$ and $X_{\rm reg}$
respectively denote the loci of the singular and the regular points of $X$.
Given a nowhere dense closed subset $Y \subset X$ such that 
$X_{\rm sing} \subseteq Y$, we let $\sZ_0(X,Y)$ denote
the free abelian group on the closed points of $X \setminus Y$.
We write $\sZ_0(X, X_{\rm sing})$ in short as $\sZ_0(X)$.

\begin{defn}\label{defn:0-cycle-S-1}
Let $C$ be a pure dimension one reduced scheme in $\Sch_k$. 
We shall say that a pair $(C, Z)$ is \emph{a good curve
relative to $X$} if there exists a finite morphism $\nu\colon C \to X$
and a closed proper subset $Z \subsetneq C$ such that the following hold.
\begin{enumerate}
\item
No component of $C$ is contained in $Z$.
\item
$\nu^{-1}(X_{\rm sing}) \cup C_{\rm sing}\subseteq Z$.
\item
$\nu$ is local complete intersection at every 
point $x \in C$ such that $\nu(x) \in X_{\rm sing}$. 
\end{enumerate}
\end{defn}

Let $(C, Z)$ be a good curve relative to $X$ and let 
$\{\eta_1, \cdots , \eta_r\}$ be the set of generic points of $C$. 
Let $\sO_{C,Z}$ denote the semilocal ring of $C$ at 
$S = Z \cup \{\eta_1, \cdots , \eta_r\}$.
Let $k(C)$ denote the ring of total
quotients of $C$ and write $\sO_{C,Z}^\times$ for the group of units in 
$\sO_{C,Z}$. Notice that $\sO_{C,Z}$ coincides with $k(C)$ 
if $|Z| = \emptyset$. 
As $C$ is Cohen-Macaulay, $\sO_{C,Z}^\times$  is the subgroup of $k(C)^\times$ 
consisting of those $f$ which are regular and invertible in the local rings 
$\sO_{C,x}$ for every $x\in Z$. 

Given any $f \in \sO^{\times}_{C, Z} \inj k(C)^{\times}$, we denote by  
${\rm div}_C(f)$ (or ${\rm div}(f)$ in short) 
the divisor of zeros and poles of $f$ on $C$, which is defined as follows. If 
$C_1,\ldots, C_r$ are the irreducible components of $C$, 
and $f_i$ is the factor of $f$ in $k(C_i)$, we set 
${\rm div}(f)$ to be the $0$-cycle $\sum_{i=1}^r {\rm div}(f_i)$, where 
${\rm div}(f_i)$ is the usual 
divisor of a rational function on an integral curve in the sense of 
\cite{Fulton}. As $f$ is an invertible 
regular function on $C$ along $Z$, ${\rm div}(f)\in \sZ_0(C,Z)$.

By definition, given any good curve $(C,Z)$ relative to $X$, we have a 
push-forward map $\sZ_0(C,Z)\xrightarrow{\nu_{*}} \sZ_0(X)$.
We shall write $\sR_0(C, Z, X)$ for the subgroup
of $\sZ_0(X)$ generated by the set 
$\{\nu_*({\rm div}(f))| f \in \sO^{\times}_{C, Z}\}$. 
Let $\sR_0(X)$ denote the subgroup of $\sZ_0(X)$ generated by 
the image of the map $\sR_0(C, Z, X) \to \sZ_0(X)$, where
$(C, Z)$ runs through all good curves relative to $X$.
We let $\CH_0(X) = \frac{\sZ_0(X)}{\sR_0(X)}$.

If we let $\sR^{LW}_0(X)$ denote the subgroup of $\sZ_0(X)$ generated
by the divisors of rational functions on good curves as above, where
we further assume that the map $\nu: C \to X$ is a closed immersion,
then the resulting quotient group ${\sZ_0(X)}/{\sR^{LW}_0(X)}$ is
denoted by $\CH^{LW}_0(X)$. Such curves on $X$ are called the 
{\sl Cartier curves}. There is a canonical surjection
$\CH^{LW}_0(X) \surj \CH_0(X)$. The Chow group $\CH^{LW}_0(X)$ was
discovered by Levine and Weibel \cite{LW} in an attempt to describe the
Grothendieck group of a singular scheme in terms of algebraic cycles.
The modified version $\CH_0(X)$ was introduced in \cite{BK}.

We shall use the following moving lemma type result from
\cite[Lemma~1.3, Corollary~1.4]{ESV} in the proof of Theorems~\ref{thm:Thm-1}
and ~\ref{thm:Thm-5}.

\begin{lem}\label{lem:MLS}
Let $X$ be a reduced quasi-projective scheme over an infinite perfect 
field $k$. Let $Y$ be a nowhere dense closed subscheme of $X$ containing 
$X_{\textnormal{sing}}$ such that the codimension of $Y$ in $X$ is at least two. 
Let $R_0^{LW}(X,Y) \subset R_0^{LW}(X)$ denote the subgroup 
generated by ${\rm div}_C(f)$ where $C$ is an integral curve such that 
$C \cap Y = \emptyset$ and $f\in k(C)^{\times}$. Then
the map 
\[
\frac{\sZ_0(X,Y)}{R_0^{LW}(X,Y)} \to \frac{\sZ_0(X)}{\sR^{LW}_0(X)}
\]
is an isomorphism.
\end{lem}

\subsection{Milnor $K$-theory sheaves}\label{sec:Milnor-sheaf}
Let $A$ be a ring. Let $T(A^{\times})$ denote the 
$\mathbb{Z}$-tensor algebra over the group of units in $A$.  
Recall that the Milnor $K$-group $K^M_i(A)$ of $A$ is the $i$-th
graded piece of the quotient of $T(A^{\times})$ by the homogeneous ideal 
generated by $a \otimes (1-a) \in A^{\times}\otimes A^{\times}$
with $a, 1-a \in A^{\times}$. 
Given an ideal $I \subset A$, we let $K^M_i(A,I) = 
{\rm Ker}(K^M_i(A) \to K^M_i(A/I))$. For $a_1, \ldots, a_i \in A^{\times}$,
we let $\{a_1, \ldots, a_i\}$ denote the image of
$a_1 \otimes \cdots\otimes a_i$ in $K^M_i(A)$.
We shall frequently use the following description of
$K^M_i(A,I)$ for local rings from \cite[Lemma~1.3.1]{Kato-Saito}.

\begin{lem}$($\cite[Lemma~1.3.1]{Kato-Saito}$)$\label{lem:Milnor-local}
Let $A$ be a finite product of local rings and let $I \subset A$ be an
ideal. Then $K^M_i(A,I)$ coincides with the subgroup of $K^M_i(A)$ generated
by elements of the form $\{a_1, \ldots,a_i\}$ such that
$a_j \in {\rm Ker}(A^{\times} \to (A/I)^{\times})$ for some $j$.
\end{lem}
 
We shall need the following local result later in the proof of
\thmref{thm:Thm-3}. Let 
\begin{equation}\label{eqn:Cart-rings}
\xymatrix@C1pc{
R \ar[r]^-{\psi_1} \ar[d]_{\psi_2} & A_1 \ar@{->>}[d]^{\phi_1} \\
A_2 \ar@{->>}[r]^-{\phi_2} & B}
\end{equation}
be a Cartesian square of rings.

\begin{lem}\label{lem:Milnor-fiber}
Associated to the Milnor square ~\eqref{eqn:Cart-rings}, the 
restriction map ${\rm Ker}(K^M_q(R) \to K^M_q(A_2)) \to
{\rm Ker}(K^M_q(A_1) \to K^M_q(B))$ is surjective if $A_1$ and $A_2$ are local 
rings. 
\end{lem}
\begin{proof}
We let $J_i = {\rm Ker}(\phi_i)$ and $I_i = {\rm Ker}(\psi_i)$ for $i =1,2$.
We need to show that the map of relative Milnor $K$-groups
$K^M_q(R, I_2) \to K^M_q(A_1, J_1)$ is surjective.
If $q \le 1$, then it follows from \cite[Theorem~6.2, Lemma~4.1]{Milnor}
that this map is actually an isomorphism. So we assume $q \ge 2$.

It is easy to check that $R$ is a local ring. It follows from
\lemref{lem:Milnor-local} that $K^M_q(A_1, J_1)$ is generated by the 
Milnor symbols $\{b_1, \ldots, b_q\}$ such that $b_j \in (1 + J_1)^{\times}$
for some $1 \le j \le q$. A similar presentation holds for $K^M_q(R, I_2)$.
We choose  such a symbol $\{b_1, \ldots, b_q\} \in K^M_q(A_1, J_1)$.
Suppose that $b_j \in (1 + J_1)^{\times}$ for some $1 \le j \le q$.
Since the map $R^{\times} \to A^{\times}_1$ is surjective, we can find
$b'_i \in R^{\times}$ such that $\psi_1(b'_i) = b_i$ for $1 \le i \neq j \le q$.
Furthermore, we have isomorphism
$(1 + I_2)^{\times} = K^M_1(R, I_2) \xrightarrow{\cong} K^M_1(A_1, J_1)
= (1 + J_1)^{\times}$ by $q =1$ case. So we can choose $b'_j \in
(1 + I_2)^{\times}$ such that $\psi^*_1(b'_j) = b_j$.
It is now immediate that $\{b'_1, \ldots, b'_q\} \in K^M_q(R, I_2)$
and $\psi^*_1(\{b'_1, \ldots, b'_q\}) = \{b_1, \ldots, b_q\}$.
This finishes the proof.
\end{proof}

\begin{defn}\label{defn:Milnor-sheaf*}
For a scheme $X$ and closed immersion $\iota: Y \inj X$,
we let $\sK^M_{i, (X,Y)}$ denote the Zariski (resp. Nisnevich) sheaf on
$X_{\zar}$ (resp. $X_{\nis}$) associated to the presheaf
$U \mapsto {\rm Ker}(K^M_i(\Gamma(\sO_U)) \to 
K^M_i(\Gamma(\sO_{Y \times_X U})))$.
\end{defn}

Since the Zariski or the Nisnevich cohomology of the
push-forward sheaf $\iota_*(\sK^M_{i,Y})$ coincides with that of
$\sK^M_{i,Y}$, we shall not distinguish between these two sheaves in the
sequel.
It follows immediately from the above definition that there is a short exact
sequence of Zariski (or Nisnevich) sheaves
\begin{equation}\label{eqn:MS*-0}
0 \to \sK^M_{i, (X,Y)} \to \sK^M_{i,X} \to \sK^M_{i,Y} \to 0.
\end{equation}

\subsection{Gersten and Cousin complexes}\label{sec:GCC}
Let $k$ be a perfect field. Let $X$ be an equi-dimensional scheme
of dimension $d$  
which is a localization of a reduced quasi-projective scheme
over $k$. For any point $x \in X$, let $K^M_i(x) = K^M_i (k(x))$.
Let $X^{(q)}$ be the set of codimension $q$ points on $X$. 
For any $x \in X^{(q)}$ and $y \in X^{(q+1)}$, let $Z = \ov{\{x\}}$.
We let $\partial^M_{x,y}: K^M_i(x) \to K^M_{i-1}(y)$ be the map
\begin{equation}\label{eqn:B-map-M}
\partial^M_{x,y} = \left\{
\begin{array}{ll}
0 & \mbox{if $y \notin Z$} \\
\sum_{z|y} N_{{k(z)}/{k(y)}} \circ \partial_z & \mbox{otherwise,}
\end{array}
\right.
\end{equation}
where $z$ runs through the closed points in $Z^N$ over $y$ and
$\partial_z: K^M_i(k(x)) = K^M_i(k(Z^N)) \to K^M_{i-1}(k(z))$ is the
classical boundary map on the quotient field of a dvr defined in \cite{BT}.

Recall from \cite[Proposition~1]{Kato} (see also \cite[Lemma~3.3]{Rost}
for a generalization)
that there is a Gersten complex of Zariski sheaves 
\begin{equation}\label{eqn:Kato-res}
0 \to \sK^M_{i, X} \xrightarrow{\epsilon}
{\underset{x \in  X^{(0)}}\amalg} (i_x)_* (K^M_i(x)) \to  
{\underset{x \in  X^{(1)}}\amalg}  (i_x)_* (K^M_{i-1}(x)) \to \cdots
\end{equation}
\[
\hspace*{5cm}
\cdots \to {\underset{x \in  X^{(d-1)}}\amalg}  (i_x)_* (K^M_{i-d+1}(x)) 
\xrightarrow{\partial^M}
{\underset{x \in  X^{(d)}}\amalg}  (i_x)_* (K^M_{i-d}(x)),
\] 
where $\epsilon$ is the usual restriction map to the
generic points. The other boundary maps consist of the 
sums of homomorphisms
$\partial^M_{x,y}$ for $x \in X^{q}, y \in X^{q+1}$.

For a Zariski sheaf $\sF$ on $X$ and a point $x \in X$
(not necessarily closed), recall that $H^q_x(X_{\zar}, \sF)$
is defined as the colimit ${\underset{U}\varinjlim} \
H^q_{\ov{\{x\}} \cap U}(U, \sF|_U)$, where the limit is over all
open neighborhoods of $x$ in $X$. The Nisnevich cohomology 
$H^q_x(X_{\nis}, \sF)$ is defined in an analogous way.

Recall also that for any Zariski sheaf of abelian groups $\sF$ on $X$, the 
filtration by codimension of support (coniveau filtration) of the Zariski
cohomology with support gives rise to the Cousin complex of
Zariski cohomology sheaves
\begin{eqnarray}\label{eqn:Cousin-res}
\nonumber C^{\bullet}_{\sF}: {\underset{x \in  X^{(0)}}\amalg} (i_x)_* H^{q}_{x}(X, \sF) &\to& 
{\underset{x \in  X^{(1)}}\amalg} (i_x)_* H^{q+1}_{x}(X, \sF) \to  \cdots \\
 &  &\cdots \to {\underset{x \in  X^{(d-1)}}\amalg}  (i_x)_* 
H^{q+d-1}_{x}(X, \sF) \xrightarrow{\partial^S} 
{\underset{x \in  X^{(d)}}\amalg}  (i_x)_* H^{q+d}_{x}(X, \sF)
\end{eqnarray}
with a map $\epsilon: \sH^q(\sF) \to \nonumber C^{\bullet}_{\sF}$, 
where $\sH^q(\sF)$ is the Zariski sheaf on $X$ associated to the presheaf
$U \mapsto H^q_{\zar}(U, \sF)$. Let
$f_S: {\underset{x \in  X^{(d)}}\amalg}  H^{d}_{x}(X, \sF)
\to H^d_{\zar}(X, \sF)$ denote the sum of the
`forget support' maps $f_{S,x}: H^{d}_{x}(X, \sF) \to
H^d_{\zar}(X, \sF)$.

For $x \in X^{(q)}$ and $y \in X^{(q+1)}$, let
$\partial^S_{x,y}: H^{n}_{x}(X, \sF) \to H^{n+1}_{y}(X, \sF)$
be the composite map
\begin{equation}\label{eqn:defn-boundary-Cousin}
H^{n}_{x}(X, \sF) \inj {\underset{z \in  X^{(q)}}\amalg} 
H^{n}_{z}(X, \sF) \xrightarrow{\partial^S}
{\underset{w \in  X^{(q+1)}}\amalg} H^{n+1}_{w}(X, \sF) 
\surj H^{n+1}_{y}(X, \sF).
\end{equation}

\begin{lem}\label{lem:Cousin-ex}
The Cousin complex induces an exact sequence of Zariski cohomology 
\begin{equation}\label{eqn:Cousin-ex-0}
{\underset{x \in  X^{(d-1)}}\amalg} 
H^{d-1}_{x}(X, \sF) \xrightarrow{\partial^S} 
{\underset{x \in  X^{(d)}}\amalg}  H^{d}_{x}(X, \sF)
\xrightarrow{f_S} H^d_{\zar}(X, \sF) \to 0.
\end{equation}
\end{lem}
\begin{proof}
The complex ~\eqref{eqn:Cousin-res} gives rise to a spectral sequence
$E^{p,q}_1 = {\underset{x \in  X^{(p)}}\amalg} H^{p+q}_{x}(X, \sF)
\Rightarrow H^{p+q}_{\zar}(X, \sF)$. 
The lemma is now an easy consequence of the fact that
$H^{i}_{x}(X, \sF) \cong H^{i}_{x}(X_x, \sF)$
for every $x \in X^{(q)}$ and $i \ge 0$ by excision, and 
$\dim(X_x \setminus \{x\}) = q-1$.
\end{proof}

The remaining part of this section is not essential for the proof of \thmref{thm:Thm-1}. The reader
who is only interested in the proof of this theorem can therefore directly move to \S~\ref{sec:CCM}.

\subsection{Higher Chow groups with modulus}\label{sec:HCGM}
For $n \ge 1$, let $\square^n$ denote the scheme 
$\A^n_k \cong (\P^1_k \setminus \{\infty\})^n$.
Let $(y_1, \cdots , y_n)$ denote the coordinate of a point on $\square^n$.
We shall denote the scheme $(\P^1_k)^n$ by $\ov{\square}^n$. 
For $1 \le i \le n$, let
$F^{\infty}_{n,i}$ denote the closed subscheme of $\ov{\square}^n$ given by the 
equation $\{y_i = \infty\}$. We shall denote the divisor 
$\stackrel{n}{\underset{i =1}\sum} F^\infty_{n,i}$ by $F^\infty_n$.

Let $X$ be a smooth quasi-projective scheme of dimension $d \ge 0$ over 
$k$ and let $D \subset X$ be an effective Cartier divisor.
For $r \in \Z$ and $n \ge 0$, let $\un{z}_r(X|D,n)$ be the free abelian group 
on integral closed subschemes $V$ of $X \times \square^n$ of dimension 
$r+n$ satisfying the following conditions.

\noindent $(1)$ (Face condition) For each face $F$ of $\square^n$, 
$V$ intersects  $X \times F$ properly:
\[
\dim_k(V\cap (X\times F)) \le r+\dim_k (F), {\rm and} 
\] 

\noindent $(2)$ (Modulus condition) $V$ is a cycle with modulus $D$ relative
to $F^{\infty}_n$:
\[
\nu^*(D \times \ov{\square}^n) \le \nu^*(X \times F^\infty_n),
\]
where $\ov{V}$ is the closure of $V$ in $X \times \ov{\square}^n$
and $\nu: \ov{V}^N \to \ov{V} \to X \times \ov{\square}^n$ is the
composite map from the normalization of $\ov{V}$. 
We let ${\un{z}_r(X|D, n)_{\dgn}}$ denote the subgroup of
$\un{z}_r(X|D, n)$ generated by cycles which are pull-back of
some cycles under various projections $X \times \square^n \to X \times
\square^{m}$ with $m < n$.

\begin{defn}\label{defn:CGM-df}
The {\em cycle complex with modulus} $(z_r(X|D, \bullet), d)$ of $X$ in 
dimension $r$ and with modulus $D$ is the non-degenerate complex associated to 
the cubical abelian group $\un{n} \mapsto \un{z}_{r}(X|D, n)$, i.e.,
\[
z_r(X|D, n): = \frac{\un{z}_r(X|D, n)}
{\un{z}_r(X|D, n)_{\dgn}}.
\]

The homology $\CH_r(X|D, n): = H_n (z_r(X|D, \bullet))$
is called a \emph{higher Chow group} of $X$ with modulus $D$.
Sometimes, we also write it as the Chow group of the {\sl modulus pair} $(X,D)$.
If $X$ has pure dimension $d$, we write 
$\CH^r(X|D, n) = \CH_{d-r}(X|D, n)$. We shall often write $\CH^r(X|D,0)$
as $\CH^r(X|D)$. We refer to \cite{KPark} for further details on this
definition. The reader should note that $\CH_r(X|D, n)$ coincides with
the usual higher Chow group of Bloch $\CH_r(X, n)$ if $D = \emptyset$.
\end{defn}

\subsection{Motivic cohomology with modulus} \label{sec:MCM}

Let $X$ be a smooth quasi-projective scheme of pure dimension $d \ge 0$ over 
$k$ and let $D \subset X$ be an effective Cartier divisor.
For an {\et}ale  map $V \to X$, we let $D_V$ denote the pull-back of $D$ to $V$. 
Then the presheaves $z^r(-|D, n)$ on the site $X_{\et}$ defined by $(V\to X) \mapsto z^r(V | D_V, n)$ 
are sheaves for the {\et}ale topology and therefore for the Nisnevich topology. 

If $X$ is smooth and $D$ is an effective Cartier divisor on $X$, then the $r$-th motivic complex of the pair $(X,D)$ is defined to be the complex of the Nisnevich sheaves on $X$:
\[
\Z(r)_{X|D} = z^r( - | D, 2r- \bullet).
\]
and the motivic cohomology of the pair $(X,D)$ is defined to be:
\[
H^i_{\sM}(X|D, \Z(r)) := \H^i_{\nis} (X, \Z(r)_{X|D}).
\]

Note that, for $0\leq r$ and $0\leq i \leq 2r$,  there exists a natural map 
\begin{equation}\label{eqn:CM-MCM}
\CH^{r}(X|D, 2r-i) \to H^i_{\sM}(X|D, \Z(r)). 
\end{equation}
Indeed, considering $z^r(X|D, 2r- \bullet)$ as a complex of  constant Nisnevich sheaves on $X$, we have 
a natural map $z^r(X|D, 2r -  \bullet) \to \Z(r)_{X|D, \nis}$. Taking the hypercohomology of these complexes, we get $ \CH^{r}(X|D, 2r-i) = \H^i_{\nis}(z^r(X|D, 2r -  \bullet))  \to H^i_{\sM}(X|D, \Z(r))$.

\subsection{The double and its Chow group}
\label{sec:double}
Let $X$ be a smooth quasi-projective scheme of dimension $d$ over $k$
and let $D \subset X$ be an effective Cartier divisor. Recall from 
\cite[\S~2.1]{BK} that the double of $X$ along $D$ is a quasi-projective
scheme $S(X,D) = X \amalg_D X$ so that
\begin{equation}\label{eqn:rel-et-2}
\begin{array}{c}
\xymatrix@R=1pc{
D \ar[r]^-{\iota} \ar[d]_{\iota} & X \ar[d]^{\iota_+} \\
X \ar[r]_-{\iota_-} & S(X,D)}
\end{array}
\end{equation}
is a co-Cartesian square in $\Sch_k$.
In particular, the identity map of $X$ induces a finite map
$\nabla:S(X,D) \to X$ such that $\nabla \circ \iota_\pm = {\rm Id}_X$
and $\pi = \iota_+ \amalg \iota_-: X \amalg X \to S(X,D)$ 
is the normalization map.
We let $X_\pm = \iota_\pm(X) \subset S(X,D)$ denote the two irreducible
components of $S(X,D)$.  We shall often write $S(X,D)$ as $S_X$
when the divisor $D$ is understood. $S_X$ is a reduced quasi-projective
scheme whose singular locus is $D_{\rm red} \subset S_X$. 
It is projective whenever
$X$ is so. It follows from \cite[Lemma~2.2]{Krishna-3} that  
~\eqref{eqn:rel-et-2}
is also a Cartesian square.

It is clear that the map
$\sZ_0(S_X,D)\xrightarrow{(\iota^*_{+}, \iota^*_{-})} \sZ_0(X_+,D) \oplus 
\sZ_0(X_-,D)$ is an isomorphism. 
Notice also that there are push-forward inclusion maps
${p_{\pm}}_* \colon \sZ_0(X,D) \to \sZ_0(S_X,D)$ such that 
$\iota^*_{+} \circ {p_{+}}_* = {\rm Id}$
and $\iota^*_{+} \circ {p_{-}}_* = 0$.
The fundamental result that connects the 0-cycles with modulus on $X$ and
0-cycles on $S_X$ is the following.

\begin{thm}$($\cite[Theorem~1.9]{BK}$)$\label{thm:BS-main}
Let $X$ be a smooth quasi-projective scheme over $k$
and let $D \subset X$ be an effective Cartier divisor. Then there
is a split short exact sequence
\[
0 \to \CH_0(X|D) \xrightarrow{ {p_{+}}_*} \CH_0(S_X) \xrightarrow{\iota^*_-}
\CH_0(X) \to 0.
\]
\end{thm}

\subsection{R{\"u}lling-Saito relative Milnor $K$-theory}

Let $k$ be a field. Let $X$ be a smooth scheme over $k$ and let $D$ be an effective Cartier divisor on $X$. Then R{\"u}lling and Saito \cite{RS} defined a variant of relative Milnor $K$-groups of the pair $(X,D)$ as follows.

For a smooth scheme $X$, we let $\hat{\sK}^{M}_{r, X}$ denote the kernel of the map of Zariski sheaves ${\underset{x \in  X^{(0)}}\amalg} (i_x)_* (K^M_i(x)) \to  
{\underset{x \in  X^{(1)}}\amalg}  (i_x)_* (K^M_{i-1}(x))$. Note that, Kerz \cite{kerz10} gave a description of these sheaves. Moreover, there exists a natural surjective map 
$\sK^M_{r, X} \surj \hat{\sK}^M_{r, X}$ which is an isomorphism at the generic points of $X$. By \cite{Kerz09}, it also follows that this map is an isomorphism if $k$ is an infinite field.

Let $j: V\hookrightarrow X$ denote the complement $X\setminus D$.  Then the Zariski sheaf $\sK^M_{r, X|D}$ is defined to be the image of the natural map 
\begin{equation} \label{eqn:def-RKMKThy}
 {\rm Ker}(\sO_X^{\times} \to \sO_D^{\times}) \otimes_{\Z} j_* \hat{\sK}^M_{r-1, V} \to j_* \hat{\sK}^M_{r,V},\
 a \otimes \{b_1, \dots, b_{r-1}\} \mapsto \{a, b_1, \dots, b_{r-1}\}. 
\end{equation}

In particular, $\sK^M_{r, X|D} = 0 $ for $r \leq 0$ and $\sK^M_{1, X|D} =  {\rm Ker}(\sO_X^{\times} \to \sO_D^{\times})$. Observe that by definition $\sK^M_{r, X|D}$ is 
the subsheaf of the Zariski sheaf ${\underset{x \in  X^{(0)}}\amalg} (i_x)_* (K^M_i(x))$.
We let $\sK^M_{r, X|D, \nis}$ denote the Nisnevich sheaf 
associated to the presheaf on the small Nisnevich site of $X$
\begin{equation} \label{eqn:RS-RMKthy-0}
X_{\nis} \to \textnormal{(abelian groups)},\  (u:  U \to X) \mapsto H^0(U, \sK^M_{r, U| u^*D}).
\end{equation}

\begin{lem}\label{lem:RS-RMKthy}
Let $k$ be an infinite filed and let $(X, D)$ be as above. Then there exists an injective map of sheaves $\sK^M_{r, (X,D), \nis} \inj \sK^M_{r, X|D, \nis}$ and it is an isomorphism if $D$ is smooth.
 Moreover, 
if the support of $D$ has a simple normal crossing, then the map induces an isomorphism of pro-sheaves $\prolim \sK^M_{r, (X,n D), \nis} \to \prolim \sK^M_{r, X| n D, \nis}$.
\end{lem}

\begin{proof} 

By the definition of the Nisnevich sheaves $\sK^M_{r, (X,D), \nis}$
and $\sK^M_{r, X|D, \nis}$  (\defref{defn:Milnor-sheaf*} and \eqref{eqn:RS-RMKthy-0}), it suffices to show that the claims of \lemref{lem:RS-RMKthy}  hold for the Zariski sheaves $\sK^M_{r, (X, D)}$ and $\sK^M_{r, X| D}$.

Since $k$ is infinite, we have $\sK^M_{*, V}  = \hat{\sK}^M_{*, V}$.
For the existence of the injective map, consider the following diagram:

\begin{equation} \label{eqn:RS-RMKthy-1}
\xymatrix@C1pc{
 {\rm Ker}(\sO_X^{\times} \to \sO_D^{\times}) \otimes_{\Z} \sK^M_{r-1,X} \ar[r] \ar[d]^{\id \otimes j^*} & \sK^M_{r,X} \ar[d]^{j^*} \ar@{^{(}->}[rd]\\
  {\rm Ker}(\sO_X^{\times} \to \sO_D^{\times}) \otimes_{\Z} j_* {\sK}^M_{r-1, V} \ar[r]&  j_* {\sK}^M_{r,V} \ar@{^{(}->}[r]& {\underset{x \in  X^{(0)}}\amalg} (i_x)_* (K^M_i(x)).}
\end{equation}

It follows easily that the square and the triangle in ~\eqref{eqn:RS-RMKthy-1} commute. By \lemref{lem:Milnor-fiber} and \cite[Lemma~2.2]{Kerz09}, the image of the top horizontal arrow in ~\eqref{eqn:RS-RMKthy-1} is the sheaf $\sK^M_{r, (X, D)}$ while by definition, the image of the bottom left horizontal arrow is the sheaf 
$\sK^M_{r, X|D}$. Therefore, there exists an inclusion of the Zariski  sheaves $\sK^M_{r, (X,D)} \inj \sK^M_{r, X|D}$.

Now, assume that the support of $D$ has simple normal crossings. Then  \cite[Proposition~2.8]{RS} gives a description of $\sK^M_{r, X|D,x}$ for $x\in D$. It follows from this description that $\sK^M_{r, X|D}$ is a subsheaf of Zariski sheaf $\sK^M_{r, X}$, and for $m \geq 2$, we have 
\[
\sK^M_{r, (X, mD)}  \inj \sK^M_{r, X|m D} \inj \sK^M_{r, (X, (m-1) D)} \inj \sK^M_{r, X}.
\]
Therefore, we have an isomorphism of pro-sheaves $\prolim \sK^M_{r, (X,n D), \nis} \to \prolim \sK^M_{r, X| n D, \nis}$.

Now, let $D$ be smooth. Let $x\in D$ and let $\eta \in X$ be the generic point of the component of $X$   which contains $x$. 
Since $D$ is smooth, it is defined at $x$ by an irreducible element 
 $t \in \sO_{X,x}$. 
Then by \cite[Proposition~2.8]{RS}, $\sK^M_{r, X|D,x}$ is the subgroup of $K^M_r(k(\eta))$ generated by the elements of the form $\{ 1+a t, u_2, \dots, u_r\}$ and $\{1+b, 1+u_1 t, u_3, \dots, u_r\} = - \{1+u_1 t,1+b, u_3, \dots, u_r \}$, where $a\in \sO_{X,x}$ and $1+b, u_i \in \sO_{X,x}^{\times}$. 
Therefore, by \lemref{lem:Milnor-local}, the natural inclusion 
$\sK^M_{r, (X,D)}  \inj \sK^M_{r, X|D}$ is an isomorphism. 
This completes the proof of the lemma. 
\end{proof}

\subsection{Relative algebraic $K$-theory}\label{sec:KTry}
Given a scheme $X$, we let $K(X)$ denote the
Bass-Thomason-Trobaugh non-connective $K$-theory spectrum of the 
biWaldhausen category of perfect complexes on $X$. This coincides with
the $K$-theory spectrum of the exact category of locally free sheaves if
$X$ is regular. We let $K_i(X)$ denote the stable homotopy groups
of the spectrum $K(X)$ for $i \in \Z$. 
Given a map $f: Y \to X$ of schemes, we let $K(X, Y)$ denote the
homotopy fiber of the map of spectra $f^*: K(X) \to K(Y)$.
If $f$ is an open immersion, we write $K(X, Y)$ as $K^{X \setminus Y}(X)$.
We let $K_i(X,Y)$ denote the stable homotopy groups
of the spectrum $K(X,Y)$ for $i \in \Z$ and we denote by $\ov{K}_i(X,Y)$
the image of the natural map $K_i(X, Y) \to K_i(X)$. 
Let $\sK_{i,X}$ denote the Zariski (or Nisnevich) sheaf on $X$
associated to the presheaf $U \mapsto K_i(U)$. The sheaves $\sK_{i, (X,Y)}$
 and $\ov{\sK}_{i,(X,Y)}$
are defined similarly.

For $X = \Spec(A)$ and an ideal $I \subset A$ with $Y = \Spec(A/I)$, 
we shall use the identifications $K(X) \cong K(A)$ and
$K(X,Y) \cong K(A,I)$. The ring structure on $K_*(A)$ 
and the natural map $K^M_1(A) = A^{\times} \to K_1(A)$ define the
maps of presheaves $\sK^M_{i,X} \to \sK_{i,X}$ and 
$\sK^M_{i,(X,Y)} \to \ov{\sK}_{i,(X,Y)}$. Let $f: Y\hookrightarrow X$ be a closed immersion 
such that $\dim(Y) < \dim(X)=d$. Since the kernel of 
the surjective map $ {\sK}_{i,(X,Y)} \to  \ov{\sK}_{i,(X,Y)}$ is supported 
on $Y$, the induced map 
$H_{\nis}^d(X, {\sK}_{i,(X,Y)}) \to H_{\nis}^d(X,  \ov{\sK}_{i,(X,Y)}) $
is an isomorphism. Therefore, there is a natural map 
$H_{\nis}^d(X, {\sK}^M_{i,(X,Y)}) \to H_{\nis}^d(X,  {\sK}_{i,(X,Y)})$.

\section{The Bloch-Quillen map for 0-cycles}\label{sec:CCM}
The Bloch-Quillen-Kato formula for smooth schemes is immediately proven
using the Gersten resolution for the Milnor and Quillen $K$-theory sheaves.
But it is not hard to see that the Gersten complex in its current form can
not give an acyclic resolution for the Milnor or Quillen $K$-theory sheaves
on singular schemes. This poses a great difficulty in proving 
analogues of the Bloch-Quillen-Kato formula for singular schemes.

Due to the lack of the Gersten resolution, the construction of a
Bloch-Quillen-Kato type map  from the Chow group to the cohomology of the Milnor
$K$-theory sheaf becomes the first major obstacle in proving 
the Bloch-Quillen-Kato formula for singular schemes.
The goal of this section is to construct this map.
The idea we use is to look at the Cousin complex instead of the
Gersten complex. It is not hard to see that this complex does give
an expression of the top cohomology of the Milnor
$K$-theory sheaf in terms of the cohomology with supports. The problem then
boils down to unraveling the appropriate boundary maps in the Cousin complex.
In the rest of this section, we show how it is achieved. 

\subsection{The map $\rho_X$ on the group of 0-cycles}
\label{sec:Def-ccm}
Let $k$ be a perfect field and let $X$ be a reduced quasi-projective scheme
of pure dimension $d \ge 0$ over $k$. Let $x \in X_{\rm reg}$ be a closed point.
We have the `forget support' map $f_{S,x}: H^d_{x}(X, \sK^M_{d,X}) \to 
H^d_{\zar}(X, \sK^M_{d,X})$ between the Zariski cohomology groups.
By \cite[Theorem 2]{Kato},
there is, for every pair of integers $n, q \ge 0$ and $x \in X^{(q)}$,
a canonical isomorphism

\begin{equation}\label{eqn:Kato-map}
\rho_x: K^M_{n-q}(k(x)) \xrightarrow{\cong} H^q_{x}(X_{\zar}, \sK^M_{n, X}).
\end{equation}

In particular, for $x \in X^{(d)}$, we have
$\Z \cong K^M_0(k(x)) \xrightarrow[\cong]{\rho_{x}}H^d_{x}(X, \sK^M_{d,X})$.
We let $\rho^{\zar}_X([x])$ denote the image of $1 \in K^M_0(k(x))$ in
$H^d_{\zar}(X, \sK^M_{d,X})$ under the forget support map. Extending this 
linearly, we obtain a map 
$\rho^{\zar}_X: \sZ_0(X) \to H^d_{\zar}(X, \sK^M_{d,X})$, which we shall
call `the Zariski Bloch-Quillen map'. Composing this
with the canonical map $H^d_{\zar}(X, \sK^M_{d,X}) \to H^d_{\nis}(X, \sK^M_{d,X})$,
we obtain our main object of study: {\sl the (Nisnevich) Bloch-Quillen map}  

\begin{equation}\label{eqn:CCM-*}
\rho_X: \sZ_0(X) \to H^d_{\nis}(X, \sK^M_{d,X}).
\end{equation}

Since $x$ is a regular point of $X$, the excision property of the cohomology
with support tells us that the map
$H^d_{x}(X, \sK_{d,X}) \to H^d_{x}(X_{\rm reg}, \sK_{d,X_{\rm reg}})$
is an isomorphism. By Gersten resolution for the Quillen $K$-theory
sheaf $\sK_{d,X_{\rm reg}}$, we 
have an isomorphism 
 $\Z \cong K_0(k(x)) \xrightarrow[\cong]{\rho_x^{\prime}} H^d_{x}(X, \sK_{d,X})$.
As before, this gives a Bloch-Quillen map  
$$\rho_X^{\prime} :  \sZ_0(X) \to H^d_{\nis}(X, \sK_{d,X}).$$

\begin{lem}\label{lem:CCM-*} With the notations as above, the diagram  
\begin{equation}\label{eqn:CCM-*-0}
\xymatrix@C.8pc{
\sZ_0(X) \ar[rr]^-{\rho_X} \ar[dr]_{\rho_X^{\prime}} & & 
H^d_{\nis}(X, \sK^M_{d,X}) \ar[dl] \\
&  H^d_{\nis}(X, \sK_{d,X}) &}
\end{equation}
is commutative, where the arrow going down on the right is induced by the
canonical map from the Milnor to Quillen $K$-theory sheaves.
\end{lem}
\begin{proof}
By definitions of $\rho_X$ and  $\rho_X^{\prime}$,
it suffices to show more generally 
that for $x\in X_{\rm reg} \cap X^{(q)}$ and $n \ge 0$, there is
a commutative diagram 
  \begin{equation} \label{eqn:CCM-*-1}
 \xymatrix@C1pc{ 
 K^M_{n-q} (k(x)) \ar[r]^-{\rho_{x}}  \ar[d]   & H^q_x(X, \sK^M_{n, X}) \ar[d] \\
  K_{n-q} (k(x)) \ar[r]^-{\rho^{\prime}_{x}} &H^q_x(X, \sK_{n, X}),}
 \end{equation}
where first row is the isomorphism of ~\eqref{eqn:Kato-map} and the
 bottom row is an isomorphism by Gersten resolution  for the Zariski
sheaf $\sK_{n,X_{\rm reg}}$ by Quillen \cite{Quillen}.

We prove that $(\ref{eqn:CCM-*-1})$ commutes by induction on $q$. 
If $q=0$ and we let $\eta_x = \Spec(k(x))$, then the terms on the right are 
$H^0(\eta_x, \sK^M_{n, \eta_x})$ and $H^0(\eta_x, \sK_{n, \eta_x})$.
Moreover, $\rho_x$ and $\rho'_x$ are defined (in \cite{Kato} and \cite{Quillen})
to be the canonical isomorphisms
$K^M_n(k(x)) \xrightarrow{\cong} H^0(\eta_x, \sK^M_{n, \eta_x})$ and
$K_n(k(x)) \xrightarrow{\cong} H^0(\eta_x, \sK_{n, \eta_x})$.
The diagram then clearly commutes.

We now assume $q \ge 1$. We let $T = (X_x)^{(q-1)}$
and consider the following diagram
  \begin{equation} \label{eqn:CCM-*-2}
 \xymatrix@C.6pc{ 
 \underset{y\in T}{\amalg}  H^{q-1}_y(X , \sK^M_{n,X})  
 \ar[rr]^{\coprod_y \partial^S_{y,x}= \partial^S_x}  
  \ar[dr] &   & H^{q}_x(X, \sK^M_{n, X}) \ar[dr]\\
 &  \underset{y\in T}{\amalg}  H^{q-1}_y(X , \sK_{n,X}) 
\ar[rr]^>>>>>>>>>{\partial^S_{x}}  &
   &   H^{q}_x(X, \sK_{n, X}) \\
  \underset{y\in T}{\amalg} K^M_{n-q+1} (k(y)) \ar[uu]^-{(\rho_y)_y} 
\ar[rd] 
\ar[rr]^>>>>>>>>>>{\partial^M_{x}} &  & K^M_{n-q}(k(x)) \ar[rd] 
\ar[uu]_>>>>>>>>>>>>>>>>{\rho_x}\\
  & \underset{y\in T}{\amalg} K_{n-q+1} (k(y)) \ar[uu]^-{(\rho^{\prime}_y)_y}
   \ar[rr]^-{\partial^Q_{x}} &  & K_{n-q}(k(x)) \ar[uu]_-{\rho^{\prime}_x}.}
\end{equation}
 
We want to show that the right face of the above cube commutes.
Since the map
$$ \partial^M_x:\underset{y\in T}{\amalg} K^M_{n-q+1} (k(y))\to K^M_{n-q}(k(x))$$
is surjective by \cite[Theorem 1]{Kato}, it suffices to show that all other 
faces of ~\eqref{eqn:CCM-*-2} commute. 
The left face commutes by induction hypothesis.
The top face commutes by the naturality of the theory of supports. 
The commutativity of the back face is part of Kato's definition
of $\rho_x$ (see \cite[\S~4]{Kato}).
The bottom face commutes because of the well known fact that 
the canonical map from the Milnor $K$-theory to the Quillen $K$-theory 
commutes with the boundary maps in the Gersten complexes on 
the regular scheme $X_x$.
Finally, the exactness of the Gersten complex for 
$\sK_{n, X_x}$ allows us to use this resolution to compute the 
boundary map in the long exact sequence for support cohomology.
It follows that under the isomorphism $\rho^{\prime}_x$, given by the 
resolution, the front face commutes.  
\end{proof}

\subsection{Compatibility with the Bloch-Quillen map  to 
$K$-theory}\label{sec:Compatibility}
Before we prove that $\rho_X$ kills the 0-cycles which are rationally
equivalent to zero, we explain how it is compatible with the cycle
class map  $\lambda_X: \sZ_0(X) \to K_0(X)$.
Recall that any regular closed point 
$x \in X$ has the property that the inclusion map $\Spec(k(x)) \inj X$
is a regular embedding. In particular, there is a push-forward map
$i_{x, *}: K_0(k(x)) \to K_0(X)$ and $\lambda_X([x])$ is the image of
$1 \in K_0(k(x))$ under this map. It is shown in \cite[Proposition~2.1]{LW}
that $\lambda_X$ factors through the rational equivalence to give a cycle
class map $\lambda_X: \CH^{LW}_0(X) \to K_0(X)$.
It is further shown in \cite[Lemma~3.13]{BK} that it factors through
the modified Chow group. Hence we have the maps

\begin{equation}\label{eqn:CCM-*-3}
\lambda_X: \CH^{LW}_0(X) \surj \CH_0(X) \to K_0(X).
\end{equation}

The Nisnevich descent spectral sequence of Thomason and 
Trobaugh \cite{TT} gives rise to a natural map
$\kappa_X: H^d_{\nis}(X, \sK^M_{d,X}) \to H^d_{\nis}(X, \sK_{d,X}) \to K_0(X)$.

\begin{lem}\label{lem:CCM-K*}
There is a commutative diagram
\begin{equation}\label{eqn:CCM-K*-0}
\xymatrix@C1pc{
\sZ_0(X) \ar[rr]^-{\rho_X} \ar[dr]_{\lambda_X} & & 
H^d_{\nis}(X, \sK^M_{d,X}) \ar[dl]^{\kappa_X} \\
& K_0(X). &}
\end{equation}
\end{lem}
\begin{proof}
By Lemma $\ref{lem:CCM-*}$, it suffices to show that ~\eqref{eqn:CCM-K*-0} 
commutes if we replace
$H^d_{\nis}(X, \sK^M_{d,X})$ by $H^d_{\nis}(X, \sK_{d,X})$. Furthermore,
we have a diagram
\begin{equation}\label{eqn:CCM-K*-1}
\xymatrix@C.8pc{
\sZ_0(X) \ar[r]^-{\rho_X^{\prime}}  \ar[dr]_{\lambda_X} & H^d_{\zar}(X, \sK_{d,X})
\ar[r] \ar[d] & H^d_{\nis}(X, \sK_{d,X}) \ar[dl] \\
& K_0(X), &}
\end{equation}
in which the triangle on the right commutes. We can therefore work with
the Zariski cohomology. Note here that the map 
$H^d_{\{x\}}(X_{\zar}, \sK^M_{d,X}) \to H^d_{\{x\}}(X_{\nis}, \sK^M_{d,X})$
is an isomorphism for $x \in X_{\rm reg}$.

We fix a closed point $x \in X_{\rm reg}$ and let $S = \Spec(k(x))$. 
We consider the diagram
\begin{equation}\label{eqn:CCM-K*-2}
 \xymatrix@C.8pc{ 
\Z \langle x \rangle  \ar[dr]_-{\cong} 
\ar[r]^-{\cong} & H^0(S, \sK_{0, S}) \ar[d]^-{\cong} \ar[r]^-{\rho'_x} 
 &  H^d_{x}(X_{\rm reg},  \sK_{d,X_{\rm reg}} )\ar[d]^-{\cong} & 
 H^d_{x}(X_{\zar},  \sK_{d,X} )\ar[d] \ar[l]_-{\cong} \ar[r] &
 H^d_{\zar}(X, \sK_{d,X}) \ar[d]\\
& K_0(S)   \ar[r]^-{i_*} &  K_0^{\{x\}}(X_{\rm reg}) & 
\ar[l]_-{\cong} K_0^{\{x\}}(X) \ar[r] & K_0(X).}
\end{equation}

By the definition of 
$\rho^{\prime}_X$, the image of $x$ in $ \Z \langle x \rangle$
maps under the composition of the top row of the diagram 
to $\rho^{\prime}_X(x) \in H^d_{\zar}(X, \sK_{d,X})$.
The composition of the bottom row sends $x$ to the element 
$\lambda_X([x]) \in K_0(X)$.
It suffices therefore to show that all squares in $(\ref{eqn:CCM-K*-2})$ 
commute. 
The middle square commutes by
the naturality of the Zariski descent spectral sequence of Thomason-Trobaugh 
for pull 
back along open immersion while the right square in $(\ref{eqn:CCM-K*-2})$ 
commutes 
by \cite[Corollaries~10.5, 10.10]{TT}. 
We are left to show that the left square in the 
diagram commutes. But this is a direct consequence of the 
comparison between the Thomason-Trobaugh and Quillen spectral
sequences for the $K$-theory of the regular scheme $X_{\rm reg}$ with support. 

Indeed, the vertical arrows in the left square in
~\eqref{eqn:CCM-K*-2} are the edge maps of the
Thomason-Trobaugh spectral sequences for $K_0(S)$ and $K^S_0(X_{\rm reg})$.
Equivalently, these are the edge maps of the
Brown-Gersten hypercohomology 
spectral sequences for $K_0(S)$ and $K^S_0(X_{\rm reg})$
(see, for example, \cite[Proof of Theorem~10.3]{TT}). 
On the other hand, it follows from \cite[Corollary 74]{Gillet} that the
Brown-Gersten hypercohomology 
spectral sequences for $K_0(S)$ and $K^S_0(X_{\rm reg})$
coincide with the corresponding Quillen spectral sequences  
from $E_2$-page  onward. So we can identify the two vertical arrows
of the left square in ~\eqref{eqn:CCM-K*-2} with the edge maps of
the Quillen spectral sequences for $K$-theory with support.
We are now done because the top horizontal arrow in this square
is induced by the push-forward map on
the Quillen spectral sequences (see \cite[\S~2.5.4, Theorem~65]{Gillet})
and the bottom horizontal arrow is the push-forward map 
on the limits of these spectral sequences.
\end{proof}

\subsection{The boundary maps in Gersten and Cousin 
complexes}\label{sec:Boundary-comp}
We shall now prove a general result which will be the
key step in the proof of the factorization of the Bloch-Quillen map 
through the rational equivalence. We begin with the following elementary but 
useful observation from commutative algebra.

\begin{lem}\label{lem:No-embedded}
If $A$ is a reduced ring, then all its associated primes are minimal.
\end{lem}
\begin{proof}
Suppose that there is a strict inclusion of associated primes
$\fp \subsetneq \fq$. Let $\ov{A} = {A}/{\fp}$
and let $\ov{\fq}$ be the image of $\fq$ in $\ov{A}$. We can write
$\fq = {\rm ann}(a)$ for some $a \in A$. Since $A$ is reduced,
it follows that $a \notin \fq$ and in particular, $\ov{a} \neq 0$.
On the other hand, $\fp \subsetneq \fq$ implies that there exists
$0 \neq \ov{b} \in \ov{\fq}$. Since $\ov{a} \ov{b} = 0$, 
we reach a contradiction as $\ov{A}$ is an integral domain.
\end{proof} 

Let $k$ be any field. Let $X$ be a
reduced quasi-projective scheme of pure dimension $d \ge 2$
over $k$. Let $n \ge 0$ be an integer.
For a Zariski sheaf $\sF$ on $X$, let 
\begin{equation}\label{eqn:Factor-0}
\partial^S: {\underset{x \in  X^{(q)}}\amalg} (i_x)_* H^{n}_{x}(X, \sF) \to
{\underset{y \in  X^{(q+1)}}\amalg} (i_y)_* H^{n+1}_{y}(X, \sF)
\end{equation}
be the boundary map of the Cousin complex ~\eqref{eqn:Cousin-res}. 

Since $\partial^S_{x,y} = 0$ if $y \notin \ov{\{x\}}$, as follows from the
construction of ~\eqref{eqn:Cousin-res}, we have a commutative diagram
(where $y \in z$ means $y \in \ov{\{z\}}$ and $\partial^S_y =
{\underset{y \in z \in  X^{(q)}}\sum} \partial^S_{z,y}$)
\begin{equation}\label{eqn:Factor-1}
\xymatrix@C.8pc{
{\underset{z \in  X^{(q)}}\amalg} H^{n}_{z}(X, \sF) \ar[r]^-{\partial^S} 
\ar@{->>}[d] & {\underset{w \in  X^{(q+1)}}\amalg} H^{n+1}_{w}(X, \sF) 
\ar@{->>}[d] 
\\
{\underset{y \in z \in  X^{(q)}}\amalg} H^{n}_{z}(X, \sF) 
\ar[r]^-{\partial^S_y} & H^{n+1}_{y}(X, \sF).}
\end{equation}

We now restrict to the case where $\sF$ is a Milnor $K$-theory sheaf.
Let $Y \subset X$ be a reduced closed subscheme and let $y \in Y^{(1)}$.
For a generic point $x$ of $Y$, let $\phi_{x,y}: K^M_n(\sO_{Y,y}) \to
K^M_n(k(x))$ be zero if $y \notin \ov{\{x\}}$ and otherwise, we let it be the
composition $K^M_n(\sO_{Y,y}) \to {\underset{y \in z \in Y^{(0)}}\amalg}
K^M_n(k(z))) \surj K^M_n(k(x)))$
along the composition $\sO_{Y,y} \inj 
{\underset{y \in z \in Y^{(0)}}\prod} k(z) \surj k(x)$.

We set
\begin{equation}\label{eqn:Factor-2}
\Phi_{Y,y} = {\underset{x \in Y^{(0)}}\amalg} \phi_{x,y}:
K^M_n(\sO_{Y,y}) \to {\underset{x \in Y^{(0)}}\amalg}  K^M_n(k(x)).
\end{equation}

Let $k$ be an infinite field. 
In this case, Kerz \cite{Kerz09} has shown that the Milnor $K$-theory sheaf on 
$X_{\rm reg}$ has a Gersten resolution. In fact, it is easy to verify that
the Gersten complex of Kerz coincides with the one defined earlier by
Kato \cite{Kato}. This implies in
particular that for a point $x \in X_{\rm reg}$, there is
a canonical isomorphism 
\begin{equation}\label{eqn:Kerz-iso}
\psi_x:K^M_{n-m}(k(x)) \xrightarrow{\cong} H^m_x(X, \sK^M_{n,X}).
\end{equation}
Moreover, the isomorphism $\psi_x$ is same as the map
$\rho_x$ in ~\eqref{eqn:Kato-map}.
We shall use this identification throughout this text.
The key step in the proof of the factorization of the Bloch-Quillen map 
through the rational equivalence is provided by the following.

\begin{prop}\label{prop:Key}
Let $X$ be a reduced quasi-projective scheme of pure dimension 
$d \ge 2$ over an infinite field and let $n \ge 0$ be an integer. Let 
$Y_{d-1} \subset Y_{d-2} \subset \cdots \subset Y_1 \subset Y_0 = X$ be a
sequence of reduced closed subschemes such that the following hold.
\begin{enumerate}
\item
$Y_i$ has pure codimension one in $Y_{i-1}$.
\item
For each $1 \le i \le d-1$, there exists a line bundle $\sL_i$ on $Y_{i-1}$ 
with a section $s_i \in \Gamma(Y_{i-1}, \sL_i)$ such that $Y_i$ is the zero-locus
of $s_i$.
\item
For each $1 \le i \le d-1$, the subset $Y_i \cap X_{\rm sing}$ is nowhere
dense in $Y_i$.
\end{enumerate}

Then for each $0 \le i \le d-1$ and $y \in Y^{(1)}_i$, the composition map
\begin{equation}\label{eqn:Key-0}
K^M_{n-i}(\sO_{Y_i, y}) \xrightarrow{\Phi_{Y_i,y}} 
{\underset{x \in Y^{(0)}_i}\amalg} K^M_{n-i}(k(x)) \xrightarrow{\cong}
{\underset{x \in Y^{(0)}_i}\amalg} H^i_{x}(X, \sK^M_{n,X}) 
\xrightarrow{\partial^S_y} H^{i+1}_{y}(X, \sK^M_{n,X})
\end{equation}
is zero.
\end{prop}
\begin{proof}
We shall prove the proposition by induction on $i$. Before we do this,
let us note that the isomorphism in the middle of ~\eqref{eqn:Key-0} is 
by ~\eqref{eqn:Kerz-iso} and our assumption (3). 
We let $\wt{\Phi}_{Y_i,y}$ denote the composition
of the middle isomorphism in ~\eqref{eqn:Key-0} with ${\Phi}_{Y_i,y}$.

{\bf STEP~1.} We let $i =0$ and fix a point $y \in X^{(1)}$. 
The long exact sequence for the cohomology with
support gives us an exact sequence (where $j_y: X^o_y \inj X_y$)
\[
H^0(X_y, \sK^M_{n, X_y}) \xrightarrow{j^*_y} H^0(X^o_y, \sK^M_{n, X^o_y})
\xrightarrow{\partial^S_y} H^1_{y}(X_y, \sK^M_{n,X_y}).
\]

We consider the diagram
\begin{equation}\label{eqn:Key-1}
\xymatrix@C.8pc{
H^0(X_y, \sK^M_{n, X_y}) \ar[r]^-{j^*_y} \ar@{=}[d] & H^0(X^o_y, \sK^M_{n, X^o_y})
\ar[r]^-{\partial^S_y} \ar@{=}[d] & H^1_{y}(X_y, \sK^M_{n,X_y}) \ar@{=}[d] \\
K^M_n(\sO_{X,y}) \ar[r]^-{(\phi_{x,y})_{y \in x}} \ar@{=}[d] & 
{\underset{y \in x \in X^{(0)}}\amalg}  H^0(\eta_x, \sK^M_{n, \eta_x})
\ar[r]^-{\partial^S_y} & H^1_{y}(X_y, \sK^M_{n,X_y}) \\
K^M_n(\sO_{X,y}) \ar[r]^-{\wt{\Phi}_{X,y}} \ar[dr]_{\Phi_{X,y}} & 
{\underset{x \in X^{(0)}}\amalg} H^0(\eta_x, \sK^M_{n, \eta_x})
\ar[r]^-{\partial^S_y} \ar@{->>}[u] &
H^1_{y}(X, \sK^M_{n,X})\ar[u]_{\cong} \\
& {\underset{x \in X^{(0)}}\amalg} K^M_n(k(x)). \ar[u]_{\cong} &}
\end{equation} 

We need to show that the composite arrow on the bottom row is zero.
But this follows because the top composite arrow is zero and 
all the squares evidently commute. We just have to observe that
$\phi_{x,y}$ is simply the restriction of $j^*_y$ to $\eta_x$, by
definition. This proves the base case $i = 0$.

{\bf Step~2.} Before we prove the proposition for $i > 0$, we 
claim that for every $i > 0$, the following hold.
\begin{enumerate}
\item
The closed subscheme $Y_i \subset Y_{i-1}$ is a Cartier divisor.
\item
$\sO_{Y_{i-1},x}$ is a discrete valuation ring for every $x \in Y^{(0)}_i
\subset Y^{(1)}_{i-1}$.
\item
Every irreducible component of $Y_i$ is contained in exactly one
irreducible component of $Y_{i-1}$.
\end{enumerate}

We let $w \in Y_i$ be a closed point and let $a_i$ be the image of 
$s_i \in \Gamma(Y_{i-1}, \sL_i)$ 
under the restriction map $\Gamma(Y_{i-1}, \sL_i) \to
\sO_{Y_{i-1}, w} \cong \Gamma(\sO_{Y_{i-1},w}, \sL_i|_{\sO_{Y_{i-1},w}})$. 
Since $Y_{i-1}$ is reduced, it follows from the assumption (1) of the 
proposition and \lemref{lem:No-embedded} that $a_i$ is a non-zero divisor in
$\sO_{Y_{i-1}, w}$ and $\sO_{Y_i, w} = {\sO_{Y_{i-1}, w}}/{(a_i)}$. 
This proves (1). If we let $w \in \ov{\{x\}}$ for any $x \in Y^{(0)}_i
\subset Y^{(1)}_{i-1}$ and let $\fp$ be the minimal prime of
$(a_i)$ defining $x$, then the assumption that $Y_i$ is reduced
implies that $\sO_{Y_{i-1},\fp}$ is an 1-dimensional local ring
whose maximal ideal $\fp\sO_{Y_{i-1},\fp}$ is generated by the image of 
$a_i$ under the localization $\sO_{Y_{i-1},w} \to \sO_{Y_{i-1}, \fp}$. 
We conclude from \cite[Theorem~11.2]{Mats}
that $\sO_{Y_{i-1},\fp}$ is a discrete valuation ring. This proves (2) and
(3) is immediate from (2) as any intersection of two or more components of
$Y_{i-1}$ is part of its singular locus. This proves the claim.

{\bf STEP~3.} We now assume $i > 0$. We fix a point $y \in Y^{(1)}_i$ and an
element $\alpha \in K^M_{n-i}(\sO_{Y_i,y})$. Since the map
$K^M_{n-i}(\sO_{Y_{i-1},y}) \to K^M_{n-i}(\sO_{Y_{i},y})$ is surjective, we can
choose a lift $\wt{\alpha}$ of $\alpha$ in $K^M_{n-i}(\sO_{Y_{i-1},y})$.  
Let $a_i$ be the image of $s_i \in \Gamma(Y_{i-1}, \sL_i)$ in 
$\sO_{Y_{i-1}, y}$ under the restriction map so that
$\sO_{Y_i, y} = {\sO_{Y_{i-1}, y}}/{(a_i)}$.  
 
For any $x \in Y_{i-1}$ such that $y \in \ov{\{x\}}$, we 
let $\wt{\alpha}_x$ be the image of $\wt{\alpha}$ under the restriction
map $K^M_{n-i}(\sO_{Y_{i-1},y}) \to K^M_{n-i}(\sO_{Y_{i-1},x})$.
We let $\ov{\alpha}_x$ be the image of $\wt{\alpha}$ under the composition
$K^M_{n-i}(\sO_{Y_{i-1},y}) \to K^M_{n-i}(\sO_{Y_{i-1},x}) \surj
K^M_{n-i}(k(x))$.

Let $\{\fp_1, \cdots, \fp_m\}$ be the set of minimal primes of
$(a_i)$ in $\sO_{Y_{i-1}, y}$. These are the generic points of $Y_i$
containing $y$. If $\fq \subset \sO_{Y_{i-1}, y}$
is a height one prime ideal such that $\fq \notin \{\fp_1, \cdots, \fp_m\}$,
then we must have $a_i \notin \fq$. It follows that 
any $x \in Y^{(1)}_{i-1} \setminus Y^{(0)}_{i}$ such that $y \in \ov{\{x\}}$, 
we have $a_i \in \sO^{\times}_{Y_{i-1},x}$. In particular, there is an element
$\wt{\beta}_{x} = a_i \cdot \wt{\alpha}_x \in K^M_{n-i+1}(\sO_{Y_{i-1},x})$.

For $y \in \ov{\{z\}}$ with $z \in Y^{(0)}_{i-1}$,  
let $a_{i, z}$ be the non-zero (as $a_i$ is
a non-zero divisor by STEP~2) image of $a_i$ in $k(z)$. 
Let $\beta_{z} = a_{i, z} \cdot \wt{\alpha}_z$ if $y \in \ov{\{z\}}$ and
zero otherwise (note that $\wt{\alpha}_z = \ov{\alpha}_z$). Set
\begin{equation}\label{eqn:Key-2}
\beta = (\beta_z)_{z \in Y^{(0)}_{i-1}} \in {\underset{z \in Y^{(0)}_{i-1}}\amalg}
K^M_{n-i+1}(k(z)).
\end{equation}

{\bf STEP~4.}
We claim that for any $x \in Y^{(1)}_{i-1}\setminus Y^{(0)}_{i}$
such that $y \in \ov{\{x\}}$, one has
$\partial^S_x(\beta) = 0$ under the map
${\underset{z \in Y^{(0)}_{i-1}}\amalg} K^M_{n-i+1}(k(z)) 
\xrightarrow{\cong} {\underset{z \in Y^{(0)}_{i-1}}\amalg} 
H^{i-1}_{z}(X, \sK^M_{n,X}) 
\xrightarrow{\partial^S_x} H^{i}_{x}(X, \sK^M_{n,X})$.

We know by the induction hypothesis that the composition
\begin{equation}\label{eqn:Key-3}
K^M_{n-i+1}(\sO_{Y_{i-1},x}) \xrightarrow{\Phi_{Y_{i-1},x}}
{\underset{z \in Y^{(0)}_{i-1}}\amalg} K^M_{n-i+1}(k(z)) 
\xrightarrow{\cong} {\underset{z \in Y^{(0)}_{i-1}}\amalg} 
H^{i-1}_{z}(X, \sK^M_{n,X}) 
\xrightarrow{\partial^S_x} H^{i}_{x}(X, \sK^M_{n,X})
\end{equation}
is zero. 

In particular, we get $\partial^S_x(\Phi_{Y_{i-1},x}(\wt{\beta}_x)) = 0$.
It suffices therefore to show that $\partial^S_x(\beta') = 0$ if
we write $\beta' = \beta - \Phi_{Y_{i-1},x}(\wt{\beta}_x)$.
Using ~\eqref{eqn:Factor-1}, we only need to show that 
$\beta'_z = 0$ if $x \in \ov{\{z\}}$. 
To prove this, we note that if $z \in Y^{(0)}_{i-1}$ is such that 
$x \in \ov{\{z\}}$, then there is a factorization
$\sO_{Y_{i-1},y} \to \sO_{Y_{i-1},x} \to \sO_{Y_{i-1},z} = k(z)$.
It follows from the above construction in this case that
$\beta_z = \phi_{z, x}(\wt{\beta}_x)$. Equivalently, $\beta'_z = 0$.
This proves the claim.

{\bf STEP~5.}
In this step, we shall study what happens to
$\partial^S_x(\beta)$  when $x \in Y^{(0)}_i \subset Y^{(1)}_{i-1}$
and $y \in \ov{\{x\}}$.
We now recall from STEP~2  that if $x \in Y^{(0)}_i \subset Y^{(1)}_{i-1}$ is 
such that $y \in \ov{\{x\}}$, then $\sO_{Y_{i-1},x}$ is a discrete valuation 
ring. It follows (see \cite{BT}) that for any $z \in Y^{(0)}_{i-1}$ with 
$x \in \ov{\{z\}}$ and $y \in \ov{\{x\}}$, 
one has $\partial^M_{x,y}(\beta_z) = \partial^M_{x,y}(a_{i,z} \cdot
\wt{\alpha}) = \ov{\alpha}_x$ under the boundary map
$\partial^M_{x,y}: K^M_{n-i+1}(k(z)) \to K^M_{n-i}(k(x))$.
Using the commutative diagram
\begin{equation}\label{eqn:Key-4}
\xymatrix@C.8pc{
K^M_{n-i}(\sO_{Y_{i-1},y}) \ar@{->>}[r] \ar[d] & K^M_{n-i}(\sO_{Y_{i},y}) 
\ar[d]^{\phi_{x,y}} \\ 
K^M_{n-i}(\sO_{Y_{i-1},x}) \ar@{->>}[r] & K^M_{n-i}(k(x)),}
\end{equation}
we get $\partial^M_{x,y}(\beta_z) = \phi_{x,y}(\alpha)$.

Let $z \in Y^{(0)}_{i-1}$ be the unique point  
such that $x \in \ov{\{z\}}$ by STEP~2. Since $k$ is infinite and
$x \in X_{\rm reg}$ by assumption (3) of the proposition, it follows from 
the Gersten resolution of $\sK^M_{n,X_{\rm reg}}$ by Kerz \cite{Kerz09} that
there is a commutative diagram
\begin{equation}\label{eqn:Key-5}
\xymatrix@C.8pc{
K^M_{n-i+1}(k(z)) \ar[r]^{\partial^M_{z,x}} \ar[d]_{\cong} &
K^M_{n-i}(k(x)) \ar[d]^{\cong} \\
H^{i-1}_{z}(X, \sK^M_{n,X}) \ar[r]^-{\partial^S_{z,x}} &
H^{i}_{x}(X, \sK^M_{n,X}).}
\end{equation}

If we identify the top and the bottom rows of ~\eqref{eqn:Key-5} 
(see ~\eqref{eqn:Kerz-iso}),
and combine this with ~\eqref{eqn:Key-4}, we see that for any
$x \in Y^{(0)}_i$ such that $y \in \ov{\{x\}}$, one has
\begin{equation}\label{eqn:Key-6}
\partial^{S}_{x}(\beta) = 
\partial^{S}_{z,x}(\beta_z) = \phi_{x,y}(\alpha) \in H^{i}_{x}(X, \sK^M_{n,X}).
\end{equation}
We note here that the first equality uses the uniqueness of $z \in Y^{(0)}_{i-1}$
such that $x \in \ov{\{z\}}$.

{\bf STEP~6.}
In the final step, we consider the commutative diagram
\begin{equation}\label{eqn:Key-7}
\xymatrix@C.8pc{
\beta \in {\underset{z \in Y^{(0)}_{i-1}}\amalg} H^{i-1}_{z}(X, \sK^M_{n,X}) 
\ar[r]^-{\partial^S} &  
{\underset{x \in Y^{(1)}_{i-1}}\amalg} H^{i}_{x}(X, \sK^M_{n,X}) 
\ar[r]^-{\partial^S_y} & H^{i+1}_{y}(X, \sK^M_{n,X}) 
\ar@{=}[d] \\
&   
\wt{\Phi}_{Y_i, y} (\alpha)
\in {\underset{x \in Y^{(0)}_{i}}\amalg} H^{i}_{x}(X, \sK^M_{n,X}) 
\ar[r]^-{\partial^S_y} \ar@{^{(}->}[u]^-{\iota}  & H^{i+1}_{y}(X, \sK^M_{n,X}).}
\end{equation}

The top row of ~\eqref{eqn:Key-7} is a complex, as one can immediately see
from the Cousin complex ~\eqref{eqn:Cousin-res}. We need to show that
$\partial^S_y(\wt{\Phi}_{Y_i, y}(\alpha)) = 0$.
Equivalently, we need to show that $\partial^S_y \circ \iota 
(\wt{\Phi}_{Y_i, y}(\alpha)) =0$. 

To show this last statement, let us write $\alpha' = 
\iota(\wt{\Phi}_{Y_i, y}(\alpha))$. It suffices to show that 
$\alpha'_x = \partial^S(\beta)_x$ for every $x \in Y^{(1)}_{i-1}$ such that
$y \in \ov{\{x\}}$.
Suppose first that $x \in Y^{(1)}_{i-1} \setminus Y^{(0)}_i$. 
In this case, $\alpha'_x$ is anyway zero and $\partial^S(\beta)_x =
\partial^S_x(\beta) = 0$ by STEP~4. If $x \in Y^{(0)}_i$, then 
$\alpha'_x = \partial^S_x(\beta)$ by ~\eqref{eqn:Key-6} in STEP~5.
This completes the proof. 
\end{proof}

\begin{remk}\label{remk:Kerz-Kato-Saito}
If we take $n = \dim(X)$, then \propref{prop:Key} and its proof remain
valid over finite fields as well in view of \cite[Theorem~2]{Kato} and
\cite[2.7.1]{Kato-Saito}.
\end{remk}   

\section{Proof of \thmref{thm:Thm-1}}\label{sec:Thm-1-prf}
We shall prove \thmref{thm:Thm-1} in this section. We begin by
showing that $\rho^{\zar}$ factors through the rational equivalence
classes.

\subsection{Factorization of $\rho^{\zar}_X$ through rational 
equivalence}\label{sec:Factor-0}
Let $k$ be an infinite perfect field and let $X$ be a reduced 
quasi-projective scheme of pure dimension $d \ge 0$ over $k$.
Recall from \S~\ref{sec:Def-ccm} that the Bloch-Quillen map 
$\rho^{\zar}_X: \sZ_0(X) \to H^d_{\zar}(X, \sK^M_{d,X})$ takes a regular
closed point $x \in X_{\rm reg}$ 
to the image of $[x] \in K_0(k(x))$ under the forget 
support map $K_0(k(x)) \cong H^d_x(X_{\zar}, \sK^M_{d,X}) \to
H^d_{\zar}(X, \sK^M_{d,X})$.

\begin{thm}\label{thm:Killing}
The Zariski Bloch-Quillen map induces a homomorphism
\begin{equation}\label{eqn:Killing-0}
\rho^{\zar}_X: \CH^{LW}_0(X) \to H^d_{\zar}(X, \sK^M_{d,X}).
\end{equation}

Moreover, the composite map $\rho_X: \CH^{LW}_0(X) \to
H^d_{\zar}(X, \sK^M_{d,X}) \to H^d_{\nis}(X, \sK^M_{d,X})$ is surjective.
\end{thm}
\begin{proof}
For $d = 1$, the theorem follows from \cite[Proposition~1.4]{LW}.
We can therefore assume that $d \ge 2$. We need to show that the
map $\rho^{\zar}_X: \sZ_0(X) \to H^d_{\zar}(X, \sK^M_{d,X})$ kills $\sR^{LW}_0(X)$.
It follows from \cite[Lemmas~1.3, 1.4]{Levine-2} that $\sR^{LW}_0(X)$
is generated by $\divf(f)$, where $C \subset X$ is a Cartier curve and
$f \in \sO^{\times}_{C, C \cap X_{\rm sing}}$ such that the following hold.
\begin{enumerate}
\item
There is a sequence of reduced closed subschemes
$C = Y_{d-1} \subset Y_{d-2} \subset \cdots \subset Y_1 \subset Y_0 = X$.
\item
For each $1 \le i \le d-1$, there is a line bundle $\sL_i$ on $Y_{i-1}$ with a 
section $s_i \in \Gamma(Y_{i-1}, \sL_i)$ such that $Y_i$ is the zero-locus
of $s_i$.
\item
$Y_i$ has pure codimension one in $Y_{i-1}$.
\item
For each $1 \le i \le d-1$, the subset $Y_i \cap X_{\rm sing}$ is nowhere
dense in $Y_i$.
\end{enumerate}

By \lemref{lem:Cousin-ex}, it suffices to show the commutativity of the
diagram
\begin{equation}\label{eqn:Killing-1}
\xymatrix@C.8pc{
{\underset{x \in X^{(d-1)}}\amalg} H^{d-1}_x(X, \sK^M_{d,X}) 
\ar[rr]^-{\partial^S} &  &
{\underset{x \in X^{(d)}}\amalg} H^{d}_x(X, \sK^M_{d,X}) 
\ar[r]^-{f_S} & H^{d}(X, \sK^M_{d,X}) \\
{\underset{x \in C^{(0)}}\amalg} H^{d-1}_x(X, \sK^M_{d,X}) \ar@{^{(}->}[u] & \\
{\underset{x \in C^{(0)}}\amalg} K^M_1(k(x)) \ar[u]_-{\cong} & 
{\underset{x \in X^{(d)}_{\rm reg}}\amalg} K^M_0(k(x)) \ar[r]^-{\cong} & 
{\underset{x \in X^{(d)}_{\rm reg}}\amalg} H^{d}_x(X, \sK^M_{d,X}) 
\ar@{^{(}->}[uu] & \\
\sO^{\times}_{C, C \cap X_{\rm sing}} \ar@{^{(}->}[u] \ar[r]^-{\divf} &
\sZ_0(X) \ar[u]_-{\cong}. & &}
\end{equation}

We let $\theta_C$ denote the composite of all vertical arrows
on the left in ~\eqref{eqn:Killing-1}. We fix a point $y \in X^{(d)}$.
It is clear that $(\partial^S \circ \theta_c)_y = 0 = (\divf)_y$
whenever $y \notin C$. So we can assume that $y \in C^{(1)}$.

Let us first assume that $y \in C \cap X_{\rm sing}$. 
We then have a commutative diagram
\begin{equation}\label{eqn:Killing-2}
\xymatrix@C.8pc{ 
\sO^{\times}_{C, C \cap X_{\rm sing}} \ar@{^{(}->}[r] \ar@{^{(}->}[d] &
{\underset{x \in C^{(0)}}\amalg} K^M_1(k(x)) \ar[r]^-{\cong} \ar@{->>}[d] & 
{\underset{x \in C^{(0)}}\amalg} H^{d-1}_x(X, \sK^M_{d,X}) \ar[r]^-{\partial^S} 
\ar@{->>}[d] & {\underset{x \in X^{(d)}}\amalg} H^{d}_x(X, \sK^M_{d,X}) 
\ar@{->>}[d] \\
K^M_1(\sO_{C,y}) \ar[r]^-{\Phi_{C,y}} & 
{\underset{y \in x \in C^{(0)}}\amalg} K^M_1(k(x)) \ar[r]^-{\cong} &
{\underset{y \in x \in C^{(0)}}\amalg} H^{d-1}_x(X, \sK^M_{d,X}) 
\ar[r]^-{\partial^S_y} & H^{d}_y(X, \sK^M_{d,X}).}
\end{equation} 

On the other hand, \propref{prop:Key} says that 
the composite of all bottom horizontal arrows in ~\eqref{eqn:Killing-2}
is zero (note that $\sO^{\times}_{C,y} \cong K^M_1(\sO_{C,y})$).
It follows that $\partial^S_y \circ \theta_C = (\partial^S \circ \theta_C)_y
= 0$. Since the map $\divf$ has support only on $X_{\rm reg}$, we also
have $(\divf)_y = 0$ so that we get $(\partial^S \circ \theta_C)_y = (\divf)_y$.

Suppose now that $y \in C^{(1)} \cap X_{\rm reg}$.
In this case, we have a diagram
\begin{equation}\label{eqn:Killing-3}
\xymatrix@C.8pc{
\sO^{\times}_{C, C \cap X_{\rm sing}} \ar@{^{(}->}[r] \ar[d]_-{\divf} & 
{\underset{x \in C^{(0)}}\amalg} K^M_1(k(x)) \ar[d]^-{\partial^M} 
\ar[r]^-{\cong} 
& {\underset{x \in C^{(0)}}\amalg} H^{d-1}_x(X, \sK^M_{d,X}) \ar@{^{(}->}[r] &
{\underset{x \in X^{(d-1)}}\amalg} H^{d-1}_x(X, \sK^M_{d,X}) 
\ar[d]^-{\partial^S_y} \\
\sZ_0(X) \ar[r]^-{\cong} & {\underset{x \in X^{(d)}_{\rm reg}}\amalg} K^M_0(k(x)) 
\ar[r]^-{\cong} & {\underset{x \in X^{(d)}_{\rm reg}}\amalg} H^{d}_x(X, \sK^M_{d,X})
\ar@{->>}[r] & H^{d}_y(X, \sK^M_{d,X}).}
\end{equation}

Since $X_{\rm reg}$ is regular, it is well known that $\partial^M$ coincides
with the divisor map. In particular, the left square commutes. The right
square commutes by the Gersten resolution of $\sK^M_{d,X}$ on $X_{\rm reg}$.
But this implies that $(\partial^S \circ \theta_C)_y = (\divf)_y$.
We have thus shown that ~\eqref{eqn:Killing-1} commutes. This shows that 
$\rho^{\zar}_X$ kills $\sR^{LW}_0(X)$.

To show that $\rho_X: \CH^{LW}_0(X) \to H^d_{\nis}(X, \sK^M_{d,X})$ is surjective,
it suffices to show that the map $\rho_X: \sZ_0(X) =
{\underset{x \in X^{(d)}_{\rm reg}}\amalg} K^M_0(k(x)) \to 
H^d_{\nis}(X, \sK^M_{d,X})$ is surjective. But this follows from
\cite[Theorem~2.5]{Kato-Saito} since $k$ is perfect and hence $U := X_{\rm reg}$
is nice in the sense of \cite[Definition~2.2]{Kato-Saito}. Moreover,
$U$ is dense in $X$. The proof of the theorem is now complete.
\end{proof}

\subsection{\thmref{thm:Thm-1} for affine schemes}\label{sec:aff**}
As a consequence of \thmref{thm:Killing}, we can now prove 
\thmref{thm:Thm-1} for affine schemes as follows.

\begin{thm}\label{thm:Main-aff}
Let $k$ be an algebraically closed field and let $X$ be a reduced 
affine scheme of pure dimension $d \ge 0$ over $k$. Then the map
\[
\rho_X: \CH^{LW}_0(X) \to H^d_{\nis}(X, \sK^M_{d,X})
\]
is an isomorphism.
\end{thm}
\begin{proof}
In view of \thmref{thm:Killing}, we only need to show that $\rho_X$ is
injective. Using \lemref{lem:CCM-K*}, it suffices to show that the
Bloch-Quillen map $\lambda_X:  \CH^{LW}_0(X) \to K_0(X)$ is injective.
For $d \le 1$, this follows from \cite[Theorem~2.3]{LW}.
For $d \ge 2$, this is \cite[Corollary~7.6]{Krishna-2}.
\end{proof}

\subsection{\thmref{thm:Thm-1} for projective schemes}
\label{sec:proj**}
We shall now prove \thmref{thm:Thm-1} for projective
schemes over an algebraically closed field which are regular in codimension
one. We fix an algebraically closed field $k$ and a reduced projective scheme
$X$ of dimension $d \ge 1$ over $k$ which is regular in codimension one.
Note that if $d = 1$, this means that $X$ is regular.
Let $\pi: X^N \to X$ denote the normalization morphism and let 
$Y = \pi^{-1}(X_{\rm sing})$.
This clearly induces the pull-back map $\pi^*: \sZ_0(X) \to 
\sZ_0(X^N, Y) \subset \sZ_0(X^N)$. 
We begin with the following reduction.

\begin{lem}\label{lem:normalization}
The map $\pi^*: \sZ_0(X) \to \sZ_0(X^N)$ induces an isomorphism
$\pi^*: \CH^{LW}_0(X) \xrightarrow{\cong} \CH^{LW}_0(X^N)$.
\end{lem}
\begin{proof}
There is nothing to prove when $d =1$ so we assume $d \ge 2$. 
By \lemref{lem:MLS}, it suffices to show that the map
$\pi^*: \sZ_0(X) \to 
\sZ_0(X^N, Y) \subset \sZ_0(X^N)$ induces an isomorphism $ \pi^*: \CH^{LW}_0(X) \to \CH^{LW}_0(X^N, Y) :=
{\sZ_0(X^N, Y)}/{\sR^{LW}_0(X^N,Y)}$.

The map $\pi^*: \sZ_0(X) \to \sZ_0(X^N,Y)$ is just the identity map. 
Furthermore, any element of $\sR^{LW}_0(X)$ is of the form $\divf(f)$, where
$C \subset X_{\rm reg}$ is an integral curve and $f \in k(C)^{\times}$.
But this uniquely defines an element of $\sR^{LW}_0(X^N,Y)$.
Conversely, any element of $\sR^{LW}_0(X^N,Y)$ is of the form
$\divf(f)$, where $C \subset X^N \setminus Y$ is an integral curve and 
$f \in k(C)^{\times}$. But $\pi(C)$ and $\pi_*(f)$ then uniquely define
an element of $\sR^{LW}_0(X)$. This proves the desired bijection
$\pi^*: \CH^{LW}_0(X) \xrightarrow{\cong} \CH^{LW}_0(X^N, Y)$.
\end{proof}  

The key step for proving \thmref{thm:Thm-1} for projective
schemes is the following result of independent interest. This result 
was proven by Levine (see \cite[Theorem~3.2]{Levine-2})
modulo $p$-torsion if ${\rm char}(k) = p > 0$. We shall follow 
Levine's outline in making his result unconditional.
The key new ingredient which makes this improvement possible and which
was not available when Levine wrote \cite{Levine-2}
is the Roitman torsion theorem for normal projective schemes in positive characteristic
\cite{KS}.

\begin{thm}\label{thm:K-inj-proj}
Let $X$ be as above.
Then the cycle class map $\lambda_X: \CH^{LW}_0(X) \to K_0(X)$ is injective.
\end{thm}
\begin{proof}
We consider the commutative diagram
\begin{equation}\label{eqn:K-inj-proj-0} 
\xymatrix@C.8pc{
{\rm CH}^{LW}_0(X) \ar[r]^-{\lambda_X} \ar[d]_{\pi^*} & K_0(X) \ar[d]^{\pi^*} \\
{\rm CH}^{LW}_0(X^N) \ar[r]^-{\lambda_{X^N}} & K_0(X^N).}
\end{equation}

It follows from \lemref{lem:normalization} that the left vertical arrow is
an isomorphism. This shows that we can assume that $X$ is normal. 
In particular, we can assume that $X$ is integral.
Levine has shown that the map $\CH^{LW}_0(X)_{\Q} \to K_0(X)_{\Q}$ is injective
(see \cite[Corollary~5.4]{Levine-5} and \cite[Corollary~2.7]{Levine-2}).
So the heart of the proof is to show that the map
$\lambda_X:\CH^{LW}_0(X)_{\rm tor} \to K_0(X)$ is injective.
Let $\CH^{LW}_0(X)_0$ denote the kernel of the degree map
${\rm deg} : \CH^{LW}_0(X) \surj \Z$. It is  clear that
$\CH^{LW}_0(X)_{\rm tor} \subset \CH^{LW}_0(X)_0$.

Recall from \cite{KS} that the normal projective variety $X$ admits
an albanese variety $A := \Alb(X)$ in the sense of \cite[Chap.~II, \S~3]{Lang} 
and an albanese rational map $u: X \dashrightarrow A$ which is regular
on $X_{\rm reg}$. If we fix a closed point $P \in X_{\rm reg}$, then
$u$ defines a surjective group homomorphism
$\tau_X: \CH^{LW}_0(X)_0 \surj A(k)$ such that 
$u(x) = \tau_X([x] -[P])$.
We shall make no distinction between $A$ and $A(k)$ in the
rest of the proof as long as the context makes it clear whether we are
talking about the variety $A$ or the group $A(k)$.

Let $X^* \subset X \times A$ be the closure of the graph of $u$ 
with projections $p : X^* \to X$ and $q: X^* \to A$.
Since $A$ is regular and $q$ is projective (because $X$ is projective),
there is a push-forward map $q_*: K_0(X^*) \to K_0(A)$
(see \cite[3.16.5]{TT}). We let $u_!: K_0(X) \to K_0(A)$ denote the
composite map $q_* \circ p^*$. Note that $u_!$ is defined on higher $K$-groups
as well, but we do not need this general version.

The morphism $p$ is an isomorphism over $X_{\rm reg}$ and $q$ agrees with 
$u$ under this isomorphism. If $x \in X_{\rm reg}$ is a closed point and
$y = p^{-1}(x)$, then we have in $K_0(A)$:
\begin{equation}\label{eqn:K-inj-proj-1} 
\begin{array}{lll}
u_! \circ \lambda_X([x]) &=& q_* \circ p^* \circ \lambda_X([x]) \\
& =& q_* \circ \lambda_{X^*}([y]) \\
&=& \sum_i [R^iq_*(k(y))] \\
&=& [q_* (k(y))] \\
& =& [k(u(x))] \\
& = & \lambda_A([u(x)]).
\end{array} 
\end{equation} 

If we identify $\sZ_0(X)$ with $\sZ_0(X^*, p^{-1}(X_{\rm sing})) \subset
\sZ_0(X^*)$, then it follows from ~\eqref{eqn:K-inj-proj-1} that for a 0-cycle 
$\alpha \in \sZ_0(X)$, one has
\begin{equation}\label{eqn:K-inj-proj-2}  
u_! \circ \lambda_X(\alpha) = \lambda_A \circ u_*(\alpha) \in K_0(A).
\end{equation}

Let $\sP$ be the Poincar\'e  line bundle on $A \times \hat{A}$, 
where $\hat{A} = \text{Pic}^0(A)$ is the dual abelian variety to $A$. 
Then $\sP$ defines a map $\smile \sP: K_0(A) \to K_0(\hat{A})$ such that 
$\smile P (\beta)  = (p_{\hat{A}})_* ( p_A^* (\beta) \otimes [\sP])$ 
for $\beta \in K_0(A)$.
If $x \in A$ is a closed point, then $\smile \sP (\lambda_A([x]))
= (p_{\hat{A}})_* (\sO_{\{x\} \times \hat{A}} \otimes [\sP])$.
Since the map $\{x\} \times \hat{A} \to \hat{A}$ is an isomorphism under
$p_{\hat{A}}$, we see that under the determinant map
${\rm det}: {\rm Ker}(K_0(\hat{A}) \xrightarrow{{\rm rk}} \Z) \to \Pic(\hat{A})$,
$\smile \sP$ restricts to 
$\smile \sP: F_0K_0(A) \to \Pic^0(\hat{A})$. Recall here that
$F_0K_0(A)$ is the subgroup of $K_0(A)$ generated by the classes of closed
points. As $(p_{\hat{A}})_* (\sO_{\{x\} \times \hat{A}} \otimes [\sP])$ is 
identified with the pull-back of $\sP$ under the embedding 
$\hat{A} \to A \times \hat{A}$, given by $y \mapsto (x,y)$,
we see furthermore that $\smile \sP (\lambda_A([x])) = x$ under 
the isomorphism $\Pic^0(\hat{A}) \cong A$ via $\sP$. We thus get a homomorphism
\begin{equation}\label{eqn:K-inj-proj-3}
\smile \sP: F_0K_0(A) \to A 
\end{equation}
such that $\smile \sP \circ \lambda_A(\alpha) = \alpha$. 

Combining the construction of
~\eqref{eqn:K-inj-proj-2} with ~\eqref{eqn:K-inj-proj-3}, one gets
a commutative diagram
\begin{equation}\label{eqn:K-inj-proj-4}
\xymatrix@C.8pc{
X_{\rm reg} \ar[r] \ar@/^2pc/[rr]^{u} 
& \CH^{LW}_0(X)_0 \ar[r]^-{\tau_X} \ar[d]_{\lambda_X} &
A \\
& F_0K_0(X) \ar[r]^-{u_!} & F_0K_0(A) \ar[u]_-{\smile \sP}. &}
\end{equation}

Since the map $\tau_X: \CH^{LW}_0(X)_{\rm tor} \to A_{\rm tor}$ is injective
by \cite[Theorem~1.6]{KS} as $X$ is normal and projective over $k$,
it follows from ~\eqref{eqn:K-inj-proj-4} that the map
$\lambda_X: \CH^{LW}_0(X)_{\rm tor} \to F_0K_0(X)_{\rm tor} \inj K_0(X)_{\rm tor}$
is also injective.
\end{proof} 

A combination of Theorems~\ref{thm:Killing}, ~\ref{thm:K-inj-proj}
and \lemref{lem:CCM-K*} yields the following result and brings us to
the end of the proof of \thmref{thm:Thm-1}.

\begin{cor}\label{cor:Proj-thm-1}
Let $X$ be a reduced projective scheme of pure dimension $d \ge 1$ over
an algebraically closed field. Assume that $X$ is regular in codimension
one. Then the Bloch-Quillen map
\[
\rho_X: \CH^{LW}_0(X) \to H^d_{\nis}(X, \sK^M_{d,X})
\]
is an isomorphism.
\end{cor}

\section{Bloch's formula for 0-cycles with modulus}
\label{sec:BF-mod}
Our goal in this section is to prove \thmref{thm:Thm-3} which provides
Bloch's formula for the Chow group of 0-cycles with modulus. We shall
do this using the double construction of \S~\ref{sec:double}
and \thmref{thm:BS-main}. We fix an algebraically closed field $k$
and a smooth quasi-projective scheme $X$ of dimension $d \ge 1$ over $k$.
We fix an effective Cartier divisor $D \subset X$. Recall that the double
of $X$ along $D$ is the scheme $S_X = X \amalg_D X$. There is a fold
map $\nabla: S_X \to X$ and inclusions as irreducible components
$\iota_\pm: X \inj S_X$ such that $\nabla \circ \iota_\pm$ is identity.

\subsection{Bloch's formula for $S_X$}\label{sec:BFD}
Our aim is to derive \thmref{thm:Thm-3} from Bloch's formula for the
singular scheme $S_X$. If $X$ is affine, this already follows from
\thmref{thm:Thm-1}. However, this is not the case when $X$ is projective.
The reason is that $S_X$ is not regular in codimension one. We shall 
now extend \thmref{thm:Thm-1} to the case of projective schemes of the type
$S_X$ under some condition on $D$.

\begin{thm}\label{thm:BQD}
Let $D \subset X$ be an inclusion of a divisor as above.
Assume that $X$ is projective and $D$ is integral. Then the Bloch-Quillen map
\[
\rho_{S_X}: \CH^{LW}_0(S_X) \to H^d_{\nis}(X, \sK^M_{d,S_X})
\]
is an isomorphism.
\end{thm}
\begin{proof}
By \thmref{thm:Killing} and \lemref{lem:CCM-K*}, we only have to show that
the cycle class map $\lambda_{S_X}: \CH^{LW}_0(S_X) \to K_0(S_X)$ is injective.
Since $\lambda_{S_X}$ is injective with $\Q$-coefficients
(see \cite[Corollary~5.4]{Levine-5} and \cite[Corollary~2.7]{Levine-2}),
the theorem is reduced to showing that the map
$\CH^{LW}_0(S_X)_{\rm tor} \to K_0(S_X)$ is injective.
We can assume $d \ge 2$ by \cite[Proposition~1.4]{LW}.
We can also assume that $X$ is connected.

We have a commutative diagram
\begin{equation}\label{eqn:BQD-0} 
\xymatrix@C.8pc{
{\rm CH}^{LW}_0(S_X) \ar[r]^-{\lambda_{S_X}} \ar[d]_{\pi^*} & 
K_0(S_X) \ar[d]^{\pi^*} \\
{\rm CH}^{LW}_0(S_X^N) \ar[r]^-{\lambda_{S_X^N}} & K_0(S_X^N).}
\end{equation}

Since $S^N_X = X_+ \amalg X_-$ is smooth and projective, the map 
$\lambda_{S^N_X}$ is injective by \cite[Theorem~3.2]{Levine-2}.
It suffices therefore to show that the map
$\pi^*: \CH^{LW}_0(S_X)_{\rm tor} \to {\rm CH}^{LW}_0(S_X^N)$ is injective.

Let $A^d(S_X)$ denote the albanese variety of $S_X$ 
and let $\tau_{S_X}: \CH^{LW}_0(S_X)_0 \to A^d(S_X)$ denote the universal
regular homomorphism (see \cite[Theorem~1]{ESV}). In general,
$A^d(S_X)$ is a connected commutative algebraic group whose
abelian variety quotient is the albanese variety of $S_X^N$ as 
in \cite{Lang}. However, under our assumption that $D$ is reduced,
it is shown in \cite[Theorem~6.5]{Krishna-3} that $A^d(S_X)$ is a 
semi-abelian variety and there exists a commutative diagram
\begin{equation}\label{eqn:BQD-1} 
\xymatrix@C.8pc{  
& & \CH^{LW}_0(S_X)_0 \ar[r]^-{\pi^*} \ar@{->>}[d]_{\tau_{S_X}} &
\CH^{LW}_0(S^N_X)_0 \ar@{->>}[d]^{\tau_{S_X}} & \\
1 \ar[r] & T \ar[r] & A^d(S_X) \ar[r]^-{\pi^*} & A^d(S^N_X) \ar[r] & 1,}
\end{equation}
where the bottom sequence is exact and $T \cong \G^r_m$ is a torus over $k$. 
Furthermore, the vertical arrows in ~\eqref{eqn:BQD-1} are isomorphisms
on the torsion subgroups by \cite[Theorem~6.5]{Krishna-3} and the
classical Roitman torsion theorem for smooth schemes.
We have therefore reduced the theorem to showing that the map
$\pi^*: A^d(S_X)_{\rm tor} \to A^d(S^N_X)_{\rm tor}$ is injective.
We shall in fact show that $T = \{1\}$ so that the map $A^d(S_X) \to A^d(S^N_X)$
is an isomorphism under the assumption that $D$ is integral.

In order to show that $T = \{1\}$, we go back to the construction of
the semi-abelian variety $A^d(S_X)$ given in \cite[\S~2.4]{Krishna-3}. 
Let ${\rm Div}(S_X)$ (resp. ${\rm Div}(S^N_X)$) denote the free abelian group 
on the set of integral closed subschemes of $S_X$ (resp. $S^N_X$) of 
codimension one. Let $\Lambda_1(S_X)$ denote the subgroup of
${\rm Div}(S^N_X)$ generated by the Weil divisors which are supported
on $D \amalg D = \pi^{-1}((S_X)_{\rm sing}) = \pi^{-1}(D)$. 
Since $\pi: D \amalg D \to D$ is the fold map of $D$, it follows that
$\Lambda_1(S_X) \cong \Z[D] \oplus \Z[D]$ and the push-forward map
$\pi_*: \Lambda_1(S_X) \to {\rm Div}(S_X)$ is the map
$\Z[D] \oplus \Z[D] \to \Z[D]$ given by $(a,b) = a + b$.

Let $NS(S^N_X) = {\Pic(S^N_X)}/{\Pic^0(S^N_X)}$ denote the
Neron-Severi group of $S^N_X$.
We have the canonical maps $\Lambda_1(S_X) \inj  {\rm Div}(S^N_X) \surj
\Pic(S^N_X) \surj NS(S^N_X)$. We let $\phi_{S^N_X}$ denote the composite map
and let $\Lambda(S_X) = {\rm Ker}(\Lambda_1(S_X) 
\xrightarrow{(\phi_{S^N_X}, \pi_*)} NS(S^N_X) \oplus {\rm Div}(S_X))$.
Then $\Lambda(S_X)$ is a lattice and $T = (\Lambda(S_X))^* \cong \G^r_m$. 
It suffices therefore to show that $\Lambda(S_X) = 0$.

We have the commutative diagram
\begin{equation}\label{eqn:BQD-2} 
\xymatrix@C.8pc{ 
\Lambda_1(S_X) \ar[r] \ar[d]_{\cong} &  \Pic(S^N_X) \ar[d]^{\cong} \ar[r] &
NS(S^N_X) \ar[d]^{\cong} \\ 
\Z[D] \oplus \Z[D] \ar[r] & \Pic(X_+) \oplus \Pic(X_-) \ar[r] &
NS(X_+) \oplus NS(X_-),}
\end{equation}
where all arrows on the bottom row are the direct sums of the
component-wise maps. We shall therefore be done if we show that the
composite map $\delta_D: \Z[D] \to \Pic(X) \surj NS(X)$ is injective.  

Let ${NS(X)}/{NS^{\tau}(X)}$ denote the group of Weil divisors on $X$ modulo 
numerical equivalence. Then one knows that there are surjections
$\Pic(X) \surj NS(X) \surj {NS(X)}/{NS^{\tau}(X)}$ and the latter group is 
free of finite rank. We have thus finally reduced the proof of the
theorem to showing that
the divisor $D$ in not numerically equivalent to zero on $X$.

We now 
choose a closed point $x \in D$ and another closed point $y \in X \setminus
D$ and let $S = \{x,y\}$. Since $X$ is connected and
smooth of dimension $d \ge 2$, 
the Bertini theorem of Altman and Kleiman \cite[Theorem~7]{AK}
implies that we can find a smooth connected curve $C \subset X$
which contains $S$. It is then clear that $C \not\subset D$ and the
intersection number $(D \cdot C)$ is positive. In particular, 
$D$ is not numerically equivalent to zero on $X$. This completes the
proof of the theorem.
\end{proof}

\subsection{Bloch-Quillen map with modulus}\label{sec:BQM-M}

Let $(X,D)$ be as before. Then for $x\in X^{(d)} \setminus D$, we have $\Z \cong  K^M_0(k(x))  \cong  H^d_{x}(V_{\nis}, \sK^M_{d,V}) \cong H^d_{x}(X_{\nis}, \sK^M_{d, (X, D)})$, where $V = X \setminus D$. 
Therefore, we have a group homomorphism 
\begin{equation} \label{eqn:BQM-M*-1}
\rho_{X|D}: z_0(X|D) \to H^d_{\nis}(X, \sK^M_{d, (X, D)})
\end{equation}
such that for $x\in X^{(d)} \setminus D$, the element $\rho_{X|D} (x)$ is the image of $1$ in  $H^d_{\nis}(X, \sK^M_{d, (X, D)})$ 
under the forget support map. 
The aim of this section is to list the pairs $(X,D)$ for which this homomorphism factors through the Chow group with modulus $\CH_0(X|D)$.  The essential idea is to use \thmref{thm:BS-main} and to use it we need to know that when is the natural map $\can_X: \CH^{LW}_0(S_X) \surj \CH_0(S_X)$ an isomorphism.
The following lemma provides all known cases in which the map $\can_X$ is indeed an isomorphism.

\begin{lem}\label{lem:LW=BK}
Let $(k, X, D)$ be as before. Then the canonical map $\can_X: \CH^{LW}_0(S_X) \surj \CH_0(S_X)$ 
is an isomorphism in the following cases.
\begin{enumerate}
\item $d\leq 2$. 
\item $X$ is affine. 
\item ${\rm char}(k) = 0 $ and $X$ is projective. 
\item ${\rm char}(k) > 0 $, $X$ is projective and $D$ is reduced.
\end{enumerate}
\end{lem}
\begin{proof}
The assertions (1)-(3) follow from \cite[Theorem~3.17]{BK} while \cite[Theorem~6.6]{Krishna-3}
yields (4). 
\end{proof}

\begin{prop} \label{prop:BQM-M}
Let $(k, X, D)$ be as in \lemref{lem:LW=BK}. Then the map in \eqref{eqn:BQM-M*-1} induces a  surjective group homomorphism 
$\rho_{X|D} : \CH_0(X|D) \surj H^d_{\nis}(X, \sK^M_{d, (X,D)})$.
\end{prop}

\begin{proof}
By \lemref{lem:LW=BK}, the canonical map
$\CH^{LW}_0(S_X) \to \CH_0(S_X)$ is an isomorphism. Combining this with
\thmref{thm:BS-main} and noting that the composite map
$X \xrightarrow{\iota_-} S_X \xrightarrow{\nabla} X$ is identity, 
we get a commutative diagram of split exact sequences
\begin{equation}\label{eqn:mod-*-0}
\xymatrix@C.8pc{
0 \ar[r] & \CH_0(X|D) \ar[r]^{p_{+ *}} \ar@{-->}[d] &
\CH^{LW}_0(S_X) \ar[r]^{\iota^*_-} \ar[d]^-{\rho_{S_X}} &
\CH_0(X) \ar[d]^-{\rho_X} \ar[r] & 0 \\
0 \ar[r] &  H^d_{\nis}(S_X, \sK^M_{d, (S_X, X_-)}) \ar[r]^-{p_{+ *}} &
H^d_{\nis}(S_X, \sK^M_{d, S_X}) \ar[r]^-{\iota^*_-} & 
H^d_{\nis}(X, \sK^M_{d, X}) \ar[r] & 0.}
\end{equation}  

This yields 
a canonical homomorphism 
$\tilde{\rho}_{X|D}: \CH_0(X|D) \to H^d_{\nis}(S_X, \sK^M_{d, (S_X, X_-)})$.
Moreover, it follows from \lemref{lem:Milnor-fiber} that the map of 
Nisnevich sheaves $\iota^*_+: \sK^M_{d, (S_X, X_-)} \to
\iota_{+ *}(\sK^M_{d, (X, D)})$ is surjective. Furthermore, its kernel
is clearly supported on $D$. It then follows from the bound on the Nisnevich
cohomological dimension that the map
$\iota^*_+: H^d_{\nis}(S_X, \sK^M_{d, (S_X, X_-)}) \to 
H^d_{\nis}(S_X, \iota_{+ *}(\sK^M_{d, (X, D)}))$ is an isomorphism.
The factorization of the map $\rho_{X|D}$ now follows because 
$H^d_{\nis}(S_X, \iota_{+ *}(\sK^M_{d, (X, D)})) \xrightarrow{\cong}
H^d_{\nis}(X, \sK^M_{d, (X, D)})$ and the  composite map 
$z_0(X|D) \to \CH_0(X|D) \xrightarrow{\tilde{\rho}_{X|D}} H^d_{\nis}(S_X, \sK^M_{d, (S_X, X_-)}) \xrightarrow{\iota^*_+} H^d_{\nis}(X, \sK^M_{d, (X,D)})$ agrees with the map 
in \eqref{eqn:BQM-M*-1}. 

To show that $\rho_{X|D}: \CH_0(X|D) \to H^d_{\nis}(X, \sK^M_{d,(X,D)})$ is surjective,
it suffices to show that the map $\rho_{X|D}: z_0(X|D) =
{\underset{x \in V^{(d)}}\amalg} K^M_0(k(x)) \to 
H^d_{\nis}(X, \sK^M_{d,(X,D)})$ is surjective. But this follows from
\cite[Theorem~2.5]{Kato-Saito} since $k$ is perfect and hence $V = X\setminus D$
is nice in the sense of \cite[Definition~2.2]{Kato-Saito}. The proof of the proposition is now complete.
\end{proof}

\subsection{Proof of \thmref{thm:Thm-3}}\label{sec:Thm-3-prf}

Now \thmref{thm:Thm-3} follows from \propref{prop:BQM-M}, 
\thmref{thm:Thm-1} and \thmref{thm:BQD}. Indeed, 
the two solid arrows in ~\eqref{eqn:mod-*-0} are isomorphisms
by \thmref{thm:Thm-1} (if $X$ is affine) or by \thmref{thm:BQD}
(if $X$ is projective and $D$ is integral). Therefore, if $(X,D)$ is 
as in  \thmref{thm:Thm-3}, then the natural map 
$\rho_{X|D} : \CH_0(X|D) \to H^d_{\nis}(X, \sK^M_{d, (X, D)})$ is an isomorphism.
This completes the proof. 
$\hfill\square$

\subsection{Proofs of Theorems~\ref{thm:CM-MCM-1} and ~\ref{thm:CM-MCM-2}}\label{sec:CM-MCM}
 Let $k$ be a field and let $X$ be a smooth quasi-projective scheme of pure dimension $d$ over $k$. Let $D$ be an effective Cartier divisor on $X$. In this section, we shall assume that 
 $D_{\red}$ is a simple normal crossing divisor, i.e., if $D_1, \dots, D_s$ are the irreducible components of $D$, then the intersections $\cap_{i \in I} D_i$ 
 are smooth over $k$ and have codimension $r$, where $I \subset \{1, \dots, s\}$ and $r$ is the cardinality of $I$. 

%

%

As an application of \thmref{thm:Thm-1}, \thmref{thm:Thm-3}, \propref{prop:BQM-M} and \cite[Theorem~3.8]{RS}, we shall now prove 
Theorems~\ref{thm:CM-MCM-1} and ~\ref{thm:CM-MCM-2}. 

Let $(X,D)$ be as in \thmref{thm:CM-MCM-1}. Then 
by \propref{prop:BQM-M}, the Bloch-Quillen map in \eqref{eqn:BQM-M*-1} induces a surjective homomorphism $\rho_{X|D} : \CH^d(X|n D) \surj
H^d(X_{\nis}, \sK^M_{d, (X, D)})$. Consider the diagram:

\begin{equation} \label{eqn:CM-MCM*-0}
\xymatrix@C3pc{
\CH^d(X|D) \ar@{->>}[r]^-{\rho_{X|D}} \ar[d]^{\can_{\sM}}&     H^d_{\nis}(X, \sK^M_{d, (X, D)})  \ar@{->>}[d]\\
H^{2d}_{\sM}(X|D, \Z(d)) \ar[r]^-{\rho_{\sM, X|D}}_{\cong} & H^d_{\nis}(X, \sK^M_{d, X| D}),}
\end{equation}
where the right vertical arrow is induced by the inclusion $ \sK^M_{d, (X, D)} \inj  \sK^M_{d, X| D}$
(see \lemref{lem:RS-RMKthy}). 
We first assume that the diagram \eqref{eqn:CM-MCM*-0} commutes and complete the proofs of Theorems~\ref{thm:CM-MCM-1} and ~\ref{thm:CM-MCM-2}.

Since the cokernel of the inclusion $\sK^M_{d,(X,D)} \inj \sK^M_{d,X|D}$
is supported on $D$, the right vertical arrow in \eqref{eqn:CM-MCM*-0} 
is surjective. By  \cite[Theorem~3.8]{RS}, the 
bottom horizontal arrow in \eqref{eqn:CM-MCM*-0}  is an isomorphism.
The surjectivity of the map $\can_{\sM}$ then follows from the commutative diagram \eqref{eqn:CM-MCM*-0}. Now assume that $D$ is connected and smooth. It then follows from ~\thmref{thm:Thm-3} that the top horizontal arrow in  \eqref{eqn:CM-MCM*-0} 
is an isomorphism.   
 Therefore, it suffices to show that the right vertical arrow in  \eqref{eqn:CM-MCM*-0} is an isomorphism. But this follows from  \lemref{lem:RS-RMKthy}, where we proved that the inclusion $\sK^M_{d, (X,D)} \inj \sK^M_{d, X|D}$ is an isomorphism. This completes the proof of \thmref{thm:CM-MCM-1} (assuming the diagram \eqref{eqn:CM-MCM*-0}
 commutes).

Now, assume that  $X$ is either affine or a quasi-projective surface. Let $D$ be an effective Cartier divisor on $X$ such that $D_{\red}$ is a simple normal crossing divisor.  Then by \thmref{thm:Thm-1} (if $X$ is affine) or 
by \cite[Theorem~1.8]{BK} (if $X$ is a quasi-projective surface), the top  horizontal arrow in ~\eqref{eqn:CM-MCM*-0} 
is an isomorphism. As before, by  \cite[Theorem~3.8]{RS}, the 
bottom horizontal arrow in ~\eqref{eqn:CM-MCM*-0}  is an isomorphism. Therefore, it suffices to show that the right vertical arrow in  \eqref{eqn:CM-MCM*-0} induces an isomorphism of pro-abelian groups. But this follows from \lemref{lem:RS-RMKthy}. This completes the proof of \thmref{thm:CM-MCM-2} (assuming 
\eqref{eqn:CM-MCM*-0} commutes).

Now, we prove that the diagram ~\eqref{eqn:CM-MCM*-0} commutes. To see this,
let $z^r(X|D, 2r-\bullet)_{c}$ denote the complex of constant Nisnevich sheaves on $X$ defined by the complex $z^r(X|D, 2r-\bullet)$.  Recall that the map $\can_{\sM}$ is induced by the morphism  of complexes of Nisnevich sheaves $z^r(X|D, 2r-\bullet)_c \to z^r( - |D, 2r-\bullet)$.
Given $x\in X\setminus D$, consider the following diagram. 


\begin{equation} \label{eqn:CM-MCM*-2}
\xymatrix@C.8pc{
\CH^d(X|D) \ar@{->>}[rrr] \ar[ddd]&  & &    H^d_{\nis}(X, \sK^M_{d, (X, D)})  \ar@{->>}[ddd] \\
& \Z\< x \> \ar[r] \ar[d]  \ar[lu] &  H^d_{x}(X_{\nis}, \sK^M_{d, (X, D)}) \ar[d] \ar[ru] \\
 & \H^{2d}_{x}(X_{\nis}, \Z(d)_{X|D}) \ar[r]  \ar[ld] &   H^d_{x}(X_{\nis}, \sK^M_{d, X| D})\ar[rd] \\
\H^{2d}_{\nis} (X, \Z(d)_{X|D}) \ar[rrr]^{\cong} & & & H^d_{\nis}(X, \sK^M_{d, X| D}).}
\end{equation}

We have to show that the front face in the diagram \eqref{eqn:CM-MCM*-2} commutes.
Since $\CH^d(X|D)$ is generated by the classes of the closed points $x\in X^{(d)}\setminus D$, it suffices to show that all the other faces in  \eqref{eqn:CM-MCM*-2} commute. The top face commutes by the definition of the Bloch-Quillen map $\lambda_{X|D}$ (see 
\eqref{eqn:BQM-M*-1}). Note here that $ \H^{2d}_{x}(X_{\nis}, z^d(X|D, 2d-\bullet)_{\nis}) = \Z\< x \>$.
The right face, the left face and the bottom face of 
\eqref{eqn:CM-MCM*-2} commute by the naturality of the `forget support'
map on the cohomology groups. Therefore, it suffices to show that the back face of 
\eqref{eqn:CM-MCM*-2} commutes. Since $x \notin D$,
we can now assume by excision that $X$ is smooth and $D = \emptyset$. In this case, the right vertical arrow is
actually the identity map (see ~\eqref{eqn:def-RKMKThy} and the proof of \lemref{lem:RS-RMKthy}). Moreover,
the commutativity of the back face is well known and follows at once from the construction of the bottom horizontal arrow
for smooth schemes (see, for instance, \cite[\S~3.1]{RS}).
This completes the proof of the Theorems~\ref{thm:CM-MCM-1} and ~\ref{thm:CM-MCM-2}.
$\hfill \square$

\section{The question of Kerz-Saito}\label{sec:BSC-prf}
We now prove \thmref{thm:Thm-5} as an application of \thmref{thm:Thm-1}.
We shall then use \thmref{thm:Thm-5} and its proof to
give a proof of \thmref{thm:Thm-7}.
Let $X$ be a reduced affine or projective scheme of pure dimension $d \ge 1$ 
over an algebraically closed field $k$ and let $\pi: \wt{X} \to X$ be
a resolution of singularities. Let $E_0 \subset \wt{X}$ be the reduced
exceptional divisor. Assume that $X$ is regular in codimension one
and let $U = X_{\rm reg} = \wt{X} \setminus E_0$. We let $S = X_{\rm sing}$
with the reduced induced closed subscheme structure and we 
let $E$ denote the scheme theoretic inverse image  $\pi^{-1}(S)$. 
Note that $E$ is supported on $E_0$ and there exists $m \geq 1$
 such that we have 
\begin{equation} \label{eqn:BSC**-(-1)}
E_0  \leq E \leq m E_0.
\end{equation}

Let $D$ be a divisor on $X$ supported on $E_0$. Given a closed point 
$x \in U$, the composite map of $K$-theory spectra
$K(k(x)) \to K(\wt{X}) \to K(D)$ is null-homotopic. 
This yields a map $u_x:K(k(x)) \to K(\wt{X}, D)$. Letting
$\lambda_{\wt{X}|D}([x]) = u_x(1) \in K_0(\wt{X},D)$, we get a cycle
class map $\lambda_{\wt{X}|D}:
\sZ_0(\wt{X}, D) \to K_0(\wt{X}, D)$. 
By the same reason, we also have a map 
$\lambda_{X|nS}: \sZ_0(X) \to K_0(X, nS)$.

It follows from
\cite[Theorem~12.4]{BK} that $\lambda_{\wt{X}|D}$ factors through 
$\lambda_{\wt{X}|D}: \CH_0(\wt{X}|D) \to K_0(\wt{X}, D)$.
We let $F_0K_0(\wt{X}, D)$ be the image of this cycle class map.
Using \lemref{lem:MLS} and Definition~\ref{defn:CGM-df}, 
it also follows easily that the map
$\pi^*:\sZ_0(X) \to \sZ_0(\wt{X}, D)$ factors through the rational
equivalence classes. We also need the following refinement of the
cycle class map $\lambda_X: \CH^{LW}_0(X) \to K_0(X)$.

\begin{lem}\label{lem:Reg-1}
The map $\lambda_{X|nS}: \sZ_0(X) \to K_0(X, nS)$ factors through the rational
equivalence classes.
\end{lem}
\begin{proof}
We let $C \subset X$ be an integral curve such that $C \cap S = \emptyset$
and let $f \in k(C)^{\times}$. By \lemref{lem:MLS}, it suffices to show that
$\lambda_{X|nS}(\divf(f)) = 0$. Let $C^N \to C$ be the normalization map
and let $\nu: C^N \to C \inj X$ denote the composite map.
It is then clear that $\divf(f) = \nu_*(\divf(f))$, where
$f \in k(C)^{\times} = k(C^N)^{\times}$.

Since $C \cap S = \emptyset$,
the finite map $\nu: C^N \to X$ has finite tor-dimension and the
resulting push-forward map $\nu_*: K(C^N) \to K(X)$ factors through
$K(C^N) \to K(X, nS) \to K(X)$ just as above.
We thus have a commutative diagram
\begin{equation}\label{eqn:BSC**-0}
\xymatrix@C.8pc{
\sZ_0(C^N) \ar[r]^-{\lambda_{C^N}} \ar[d]_-{\nu_*} & K_0(C^N) \ar[d]^-{\nu_*} \\
\sZ_0(X) \ar[r]^-{\lambda_{X|nS}} & K_0(X,nS).}
\end{equation}
We are now done since $\lambda_{C^N}(\divf(f)) = 0$.
\end{proof}

\subsection{Proof of \thmref{thm:Thm-5}}\label{sec:BS-prf}
We shall now prove \thmref{thm:Thm-5}. Using \lemref{lem:Reg-1} and the construction of
various other maps before it, we obtain a commutative diagram for every 
$n \ge 1$:

\begin{equation}\label{eqn:BSC**-1}
\xymatrix@C.8pc{
& \CH^{LW}_0(X) \ar@{->>}[r]^-{\pi^*} \ar@{->>}[d]^-{\lambda_{X|nS}} 
\ar@{->>}[dl]_-{\lambda_X}^{\cong} &
\CH_0(\wt{X}|nE) \ar@{->>}[d]^-{\lambda_{\wt{X}|nE}} \ar@{->>}[r] & 
\CH_0(\wt{X}|nE_0) \ar@{->>}[r] \ar@{->>}[d]^-{\lambda_{\wt{X}|nE_0}}  & 
\CH_0(\wt{X}) \ar@{->>}[d]^-{\lambda_{\wt{X}}} \\
F_0K_0(X) & F_0K_0(X, nS) \ar@{->>}[r]^-{\pi^*} \ar@{->>}[l] & 
F_0K_0(\wt{X}, nE) \ar@{->>}[r] & 
F_0K_0(\wt{X}, nE_0)  \ar@{->>}[r] & F_0K_0(\wt{X}).}
\end{equation}

The map $\lambda_X$ on the left is an isomorphism by
\cite[Corollary~7.6]{Krishna-2} (if $X$ is affine) and 
\thmref{thm:K-inj-proj} (if $X$ is projective).
It follows that all arrows in the triangle on the left are isomorphisms.
By \cite[Theorem A]{KST}, the canonical homomorphism of pro-abelian groups 
$\prolim K_0(X, nS) \to \prolim K_0(\wt{X}, nE)$ is an isomorphism. 
In particular, its restriction  $\prolim F_0K_0(X, nS) \to \prolim F_0K_0(\wt{X}, nE)$
is an isomorphism too. By ~\eqref{eqn:BSC**-(-1)}, it follows that the map of 
pro-abelian groups 
$\prolim F_0K_0(\wt{X}, nE) \to \prolim F_0K_0(\wt{X}, nE_0)$
is an isomorphism.
As $\lambda_{X|nS}$ is an isomorphism for all $n \ge 1$,
we get an isomorphism of pro-abelian groups 
$\CH^{LW}_0(X) \xrightarrow{\cong} \prolim F_0K_0(\wt{X}, nE_0)$.
Since $\CH^{LW}_0(X)$ is a constant  pro-abelian group, an elementary
calculation shows that we must have
$\CH^{LW}_0(X) \xrightarrow{\cong} F_0K_0(\wt{X}, nE_0)$ for all $n \gg 1$.
It follows that all arrows in the left square and in the middle square  of  
~\eqref{eqn:BSC**-1} are isomorphisms for all $n \gg 1$. 
This proves the standard version of the Bloch-Srinivas conjecture
(part (1) of \thmref{thm:Thm-5}).

We assume now that ${\rm char}(k) = p > 0$ and prove the strong version,
namely, that $\CH^{LW}_0(X) \xrightarrow{\cong} \CH_0(\wt{X}|E_0)
\xrightarrow{\cong} F_0K_0(\wt{X}, E_0)$.

Using the homotopy fiber sequence of spectra
\[
K(\wt{X}, nE_0) \to K(\wt{X}, E_0) \to K(nE_0, E_0)
\]
and \cite[Corollary~5.4]{Weibel}, it follows that the kernel of the map
$F_0K_0(\wt{X}, nE_0) \surj F_0K_0(\wt{X}, E_0)$ is a $p$-primary torsion group.
We remark here that the cited reference is enough only for affine schemes. But one can then use
a Mayer-Vietoris argument to prove it in our case too. The reader can also see 
\cite[Lemma~3.4]{Krishna-2}, which shows a stronger result that this kernel 
is in fact a $p$-primary torsion group of bounded exponent. 

We choose $n \gg 1$ such that 
$\CH^{LW}_0(X) \xrightarrow{\cong} \CH_0(\wt{X}|nE_0) \cong
F_0K_0(\wt{X}, nE_0)$. It follows then that the kernel of the 
composite map $\CH^{LW}_0(X) \surj \CH_0(\wt{X}|E_0) \surj F_0K_0(\wt{X}, E_0)$ 
is a $p$-primary torsion group of bounded exponent.

If $X$ is affine, this kernel must be zero by \cite[Theorem~1.1]{Krishna-2}.
If $X$ is projective, we have a commutative diagram
\begin{equation}\label{eqn:BSC**-2}
\xymatrix@C.8pc{
& \CH^{LW}_0(X)_{\rm tor} \ar[r] \ar[d]^-{\cong} \ar[ddl]_{\cong}
& \CH_0(\wt{X}|E_0)_{\rm tor} \ar[r] \ar[d] & \CH_0(\wt{X})_{\rm tor} 
\ar[d]^-{\cong} \\
& F_0K_0(X)_{\rm tor} \ar[r] \ar[d]^-{\cong}  & 
F_0K_0(\wt{X}, E_0)_{\rm tor} \ar[r] & F_0K_0(\wt{X})_{\rm tor}
\ar[d]_-{\tau_{\wt{X}}}^-{\cong} \\
\CH^{LW}_0(X^N)_ {\rm tor} \ar[r]^-{\cong} & F_0K_0(X^N)_ {\rm tor}
\ar[r]_-{\tau_{X^N}}^-{\cong} & 
A^d(X^N)_{\rm tor} \ar[r]^-{\cong} & A^d(\wt{X})_{\rm tor}.}
\end{equation}

All arrows in the left triangle are isomorphisms by
\lemref{lem:normalization} and \thmref{thm:K-inj-proj}.
By the same reason, the vertical arrow on the top right is an
isomorphism.
It follows from this diagram and \cite[Theorem~1.6]{KS} that the composite map
$\CH^{LW}_0(X)_{\rm tor} \to   F_0K_0(\wt{X}, E_0)_{\rm tor} \to
F_0K_0(\wt{X})_{\rm tor} \to A^d(\wt{X})_{\rm tor}$ is an 
isomorphism.
In particular, the map $\CH^{LW}_0(X)_{\rm tor} \surj 
F_0K_0(\wt{X}, E_0)_{\rm tor}$ is injective. It follows that 
${\rm Ker}(\CH^{LW}_0(X) \surj F_0K_0(\wt{X}, E_0))$  must be zero.
The proof of \thmref{thm:Thm-5} is now complete.
$\hfill \square$

\subsection{Proof of \thmref{thm:Thm-7}}\label{sec:KSQ-prf}
We now prove \thmref{thm:Thm-7}. We shall follow the notations of
the statement of \thmref{thm:Thm-7} in its proof. 
Recall from \thmref{thm:Thm-7} that
$Y$ is a reduced projective scheme of pure dimension $d$
over an algebraically
closed field $k$ of positive characteristic. Our assumption is that
$Y$ is regular in codimension one and $\pi: X \to Y$
is a resolution of singularities with the reduced exceptional divisor 
$E_0 \subset X$. Let $S \subset Y$ be the singular locus with the
reduced subscheme structure. Moreover, let $E$ denote the scheme
 theoretic inverse image  $\pi^{-1}(S)$. Note that these notations 
 are a little different from the ones in
\thmref{thm:Thm-5}.

We fix an integer $n \ge 1$ and consider the commutative diagram
\begin{equation}\label{eqn:KES-0}
\xymatrix@C.8pc{
\CH^{LW}_0(Y) \ar@/_1pc/[ddr]_-{\rho_Y} \ar@{..>}[dr] 
\ar@{->>}[drr]^-{\lambda_{Y|nS}} & &  \\
&
H^d_{\nis}(Y, \sK^M_{d, (Y, nS)}) \ar[d]_-{\cong} \ar@{->>}[r] & F_0K_0(Y, nS) 
\ar[d]^-{\cong} \\
& H^d_{\nis}(Y, \sK^M_{d, Y}) \ar[r]^-{\cong} & F_0K_0(Y).}
\end{equation}

The left vertical arrow on the bottom square is an isomorphism as
$\dim(S) \le d-2$. We have shown in the proof of
\thmref{thm:Thm-5} that all solid arrows in ~\eqref{eqn:KES-0} are
isomorphisms. It follows that the map $\rho_{Y|nS}:\sZ_0(Y) \to 
H^d_{\nis}(Y, \sK^M_{d, (Y, nS)})$ factors through the Chow group 
$\CH^{LW}_0(Y)$ so that ~\eqref{eqn:KES-0} is commutative and all maps
are isomorphisms.

We next consider the commutative diagram
\begin{equation}\label{eqn:KES-1}
\xymatrix@C.8pc{
\sZ_0(Y) \ar@{->>}[r] \ar[d]_-{\cong} & \CH^{LW}_0(Y) 
\ar[r]^-{\rho_{Y|nS}}_-{\cong} \ar[d]^-{\cong}_-{\pi^*} & 
H^d_{\nis}(Y, \sK^M_{d, (Y, nS)}) \ar[r]_-{\cong}^-{\theta_{Y|nS}} \ar[d]^-{\pi^*} 
& F_0K_0(Y, nS) \ar[d]^-{\pi^*}_-{\cong} \ar[r]^-{\cong} 
& F_0K_0(Y,S) \ar[d]^-{\pi^*}_-{\cong} \\
\sZ_0(X|nE) \ar@{->>}[r] & \CH_0(X|nE) \ar@{..>}[r]^-{\rho_{X|nE}} &
H^d_{\nis}(X, \sK^M_{d, (X, nE)}) \ar@{->>}[r]^-{\theta_{X|nE}} & F_0K_0(X,nE)
\ar[r]^-{\cong} & F_0(X, E),}
\end{equation}
where $\theta_{X|nE}$ is the composition of the edge map in the 
Thomason-Trobaugh spectral sequence with natural map
$H^d_{\nis}(X, \sK^M_{d, (X, nE)}) \to H^d_{\nis}(X, {\sK}_{d, (X, nE)})$
of \S~\ref{sec:KTry}.

We have shown in the proof of \thmref{thm:Thm-5} that all solid arrows
in ~\eqref{eqn:KES-1} (except possibly the middle vertical arrow)
are isomorphisms. A simple diagram chase shows that 
the dotted arrow $\rho_{X|nE}$ is in fact a solid arrow.
We now consider the commutative diagram
\begin{equation}\label{eqn:KES-1.5}
\xymatrix@C2pc{
\CH^{LW}_0(Y) \ar@{->>}[r]^{\pi^*} \ar@{->>}[rd]_{\pi^*} & 
\CH_0(X|nE) \ar[r]^-{\rho_{X|nE}}  \ar@{->>}[d]&
H^d_{\nis}(X, \sK^M_{d, (X, nE)}) \ar@{->>}[r]^-{\theta_{X|nE}} 
\ar[d] & F_0K_0(X,nE)\ar[d]\\
&  \CH_0(X|nE_0) \ar@{..>}[r]^-{\rho_{X|nE_0}} &
H^d_{\nis}(X, \sK^M_{d, (X, nE_0)}) 
\ar@{->>}[r]^-{\theta_{X|nE_0}} & F_0K_0(X,nE_0)
,}
\end{equation}
where vertical arrows exist as $E_0 \subset E$.  
By \thmref{thm:Thm-5}, it follows that the left 
vertical arrow in ~\eqref{eqn:KES-1.5} is an isomorphism. 
A simple diagram chase shows that 
the dotted arrow $\rho_{X|nE_0}$ is in fact a solid arrow. Since we 
proved in \thmref{thm:Thm-5}
that the composite map $\CH^{LW}_0(Y) \to  F_0K_0(X,E_0)$
is an isomorphism, it follows the map $\rho_{X|nE_0}$ is injective.
 On the other hand, the composite map $\sZ_0(X|nE_0) \to \CH_0(X|nE_0) 
\xrightarrow{\rho_{X|nE_0}} H^d_{\nis}(X, \sK^M_{d, (X, nE_0)})$
is surjective by \cite[Theorem~2.5]{Kato-Saito}.
We conclude that  all arrows in ~\eqref{eqn:KES-1.5}
are isomorphisms. In particular,  $\rho_{X|nE_0}$
  is an isomorphism for every $n \ge 1$. 

To finish the proof of \thmref{thm:Thm-7}, we let $D \subset X$ be
any effective Cartier divisor with support $E_0$.
We can then find two inclusions $E_0 \subset D \subset nE_0$ for some
$n \gg 0$.
This gives rise to a commutative diagram
\begin{equation}\label{eqn:KES-2}
\xymatrix@C.8pc{
\CH^{LW}_0(Y) \ar[r]^-{\cong} \ar[d]_-{\cong} & 
\CH_0(X|nE_0) \ar[r]^-{\cong} \ar[d]^-{\cong} &
\CH_0(X|D) \ar[r]^-{\cong} \ar@{..>}[d] & \CH_0(X|E_0) \ar[d]^-{\cong} \\
H^d_{\nis}(Y, \sK^M_{d,Y}) \ar[r] & H^d_{\nis}(X, \sK^M_{d, (X, nE_0)}) \ar@{->>}[r] &
H^d_{\nis}(X, \sK^M_{d, (X, D)}) \ar@{->>}[r] & H^d_{\nis}(X, \sK^M_{d, (X, E_0)}).}
\end{equation}

A diagram chase shows that all solid arrows in ~\eqref{eqn:KES-2} are
isomorphisms. This implies that the vertical dotted arrow is in fact 
a solid arrow and is an isomorphism. In other words, the Bloch-Quillen map
$\sZ_0(X|D) \to H^d_{\nis}(Y, \sK^M_{d,Y})$ induces an
isomorphism $\rho_{X|D}: \CH_0(X|D) \xrightarrow{\cong} 
H^d_{\nis}(X, \sK^M_{d,(X,D)})$. We have thus proven \thmref{thm:Thm-7}.  
$\hfill \square$

\section{Euler class groups of affine algebras}\label{sec:ECG}
In order to prove Theorems~\ref{thm:Thm-2}, ~\ref{thm:Thm-4} and 
~\ref{thm:Thm-6}, we shall use the theory of Euler class groups of 
affine algebras. The Euler class group of a $k$-algebra $A$ 
has an advantage that any class in this group is the class of a nice
enough ideal $J \subset A$ which has a class $[A/J]$ in $K_0(A)$ as well.
If this class dies in $K_0(A)$, then there are some commutative algebra
results which allow us to conclude that the class of $J$ is zero in the
Euler class group as well. So the key to proving a result like 
Theorem~\ref{thm:Thm-6} is to connect the Levine-Weibel Chow group
with these Euler class groups.

Unfortunately, the Euler class group has cycles which are supported on the
singular locus of $\Spec(A)$, and hence, it is very hard to directly connect
this group with the Chow group. To circumvent this
problem, we introduce a new version of the Euler class group. 
This new version is closely related to the Chow group.
The key result of this section is that this new version is canonically
isomorphic to the original one. This will be used in the next 
section to finish the proofs of Theorems~\ref{thm:Thm-2}, ~\ref{thm:Thm-4} and 
~\ref{thm:Thm-6}. 

Throughout this section, we fix a 
 field $k$ and all rings we consider will be geometrically reduced 
equi-dimensional affine algebras over $k$.  

\subsection{The Euler class groups}\label{sec:ECG*}
We recall the definitions of the Euler class groups from \cite{BS-3}.
Let $A$ be an affine $k$-algebra of dimension $d \ge 2$. 
Let $G(A)$ be the free abelian group on the pairs $(\fn, \omega_{\fn})$, where
$\fn \subset A$ is an $\fm$-primary ideal for a maximal ideal $\fm \subset A$
of height $d$ and $\omega_{\fn}: (A/{\fn})^{d} \surj {\fn}/{\fn^2}$ is 
an $A$-linear surjection.

Given an ideal $J \subset A$ of height $d$ with 
the irredundant primary decomposition $J = \fn_1 \cap \cdots \cap \fn_r$
and a surjection $\omega_J: (A/{J})^d \surj {J}/{J^2}$, the Chinese remainder 
theorem yields surjections $\omega_{\fn_i}: (A/{\fn_i})^d \surj 
{\fn_i}/{\fn^2_i}$.
In particular, the ideal $J$ and the map $\omega_J$
together define a unique class $(J, \omega_J) = 
\stackrel{r}{\underset{i =1}\sum} (\fn_i, \omega_{\fn_i}) \in G(A)$.
Let $H(A) \subset G(A)$ be the subgroup generated by the classes $(J, \omega_J)$
as above such that there exists an $A$-linear surjective homomorphism $\wt{\omega_J} : A^d \surj J$
and a commutative diagram of $A$-modules:
\begin{equation}\label{eqn:ECG-0}
\xymatrix@C1pc{
A^d \ar@{->>}[r]^{\wt{\omega}_J} \ar[d] & J \ar[d] \\
(A/J)^d \ar@{->>}[r]_{\omega_J} & J/{J^2}.}
\end{equation}
The Euler class group of $A$ is defined to be the group $E(A) = {G(A)}/{H(A)}$.

The weak version of the Euler class group is defined as follows. 
Let $G_0(A)$ denote the free abelian group on the set of ideals $\fn \subset A$
such that $\fn$ is an $\fm$-primary ideal for some maximal ideal of
height $d$ in $A$ and there is a surjective $A$-linear map
$\omega_{\fn}: (A/{\fn})^d \surj {\fn}/{\fn^2}$. 

Given an ideal $J \subset A$ of height $d$ with 
the irredundant primary decomposition $J = \fn_1 \cap \cdots \cap \fn_r$
and a surjection $\omega_J: (A/{J})^d \surj {J}/{J^2}$, the Chinese remainder 
theorem yields surjections $\omega_{\fn_i}: (A/{\fn_i})^d \surj 
{\fn_i}/{\fn^2_i}$.
In particular, the ideal $J$ defines a unique class $(J) = 
\stackrel{r}{\underset{i =1}\sum} \fn_i \in G_0(A)$.
Let $H_0(A) \subset G_0(A)$ be the subgroup generated by the classes $(J)$ 
such that $J$ is generated by $d$ elements. The weak Euler class group is 
defined as $E_0(A) = {G_0(A)}/{H_0(A)}$.

It is clear that there is a canonical {\sl forget orientation} 
map $\psi_A: E(A) \surj E_0(A)$.

\subsection{The Segre exact sequence}\label{sec:Segre}
Given a positive integer $n$, let $Um_n(A)$ denote
the set of unimodular rows of length $n$ over $A$. Recall here that
a row $\underline{a}:= [a_1, \ldots, a_n] \in M_{1,n}(A)$ is called 
unimodular, if the ideal $(a_1, \ldots, a_n)$ is $A$.
If $B \in M_{n,1}$ is such that $\underline{a}B = 1$, then
we have $1 = \underline{a}B =\underline{a}M M^{-1}B$ for any $M \in GL_n(A)$.
Setting $B' = M^{-1}B$, we get $(\underline{a}M)B' = 1$. Using this, one
can easily show  that $GL_n(A)$ acts on $Um_n(A)$. We let
$WS_n(A) = {Um_n(A)}/{E_n(A)}$ be the quotient for the action of 
the elementary matrices $E_n(A)$ on $Um_n(A)$.
It was shown by van der Kallen \cite{vdK-1} that
$WS_n(A)$ is an abelian group. 
For any $\underline{a} \in Um_n(A)$, let $[\underline{a}]$ denote its
equivalence class in $WS_n(A)$.
We now quote the following independent results of Das-Zinna \cite{DZ} and
van der Kallen \cite{vdK-2}. When $d \ge 2$ is even and $\Q \subset A$, 
this was earlier
proven by Bhatwadekar-R. Sridharan \cite[Theorem~7.6]{BS-3}.

\begin{thm}\label{thm:Segre}
There is an exact sequence
\begin{equation}\label{eqn:Segre-0}
WS_{d+1}(A) \xrightarrow{\phi_A} E(A) \xrightarrow{\psi_A} E_0(A) \to 0.
\end{equation}
\end{thm}

\subsection{The modified Euler class groups}\label{sec:MECG}
We now introduce the modified Euler class groups. 
We shall say that an ideal
$J \subset A$ is {\sl regular}, if it is reduced (i.e., $J = \sqrt{J}$)
and the localization $A_{\fp}$ is a regular local ring for every minimal
prime $\fp$ of $J$. For any finitely generated $A$-module $M$,
let $\mu(M)$ denote the smallest positive integer $m$ such that $M$
is generated by $m$ elements.

(1) Let $G^s(A)$ denote the free abelian group on the set of pairs
$(\fm, \omega_{\fm})$, where $\fm \subset A$ is a regular 
maximal ideal of height $d$ and $\omega_{\fm}: (A/{\fm})^d \to {\fm}/{\fm^2}$
is an isomorphism. 

Given a regular ideal $J \subset A$ of height $d$ with 
the primary decomposition $J = \fm_1 \cap \cdots \cap \fm_r$
and an isomorphism $\omega_J: (A/{J})^d \xrightarrow{\simeq} {J}/{J^2}$, 
the Chinese remainder theorem yields isomorphisms
$\omega_{\fm_i}: (A/{\fm_i})^d \xrightarrow{\simeq} {\fm_i}/{\fm^2_i}$.
In particular, the ideal $J$ and the map $\omega_J$ together
define a unique class $(J, \omega_J) = 
\stackrel{r}{\underset{i =1}\sum} (\fm_i, \omega_{\fm_i}) \in G^s(A)$.
Let $H^s(A) \subset G^s(A)$ be the subgroup generated by the classes 
$(J, \omega_J)$
as above such that there is a commutative diagram of $A$-modules:
\begin{equation}\label{eqn:ECG-0*}
\xymatrix@C1pc{
A^d \ar@{->>}[r]^{\wt{\omega}_J} \ar[d] & J \ar[d] \\
(A/J)^d \ar[r]^{\simeq}_{\omega_J} & J/{J^2}.}
\end{equation}
We let $E^s(A) = {G^s(A)}/{H^s(A)}$. 

(2) Let $G^s_0(A)$ denote the free abelian group on the set of regular
maximal ideals $\fm \subset A$ of height $d$.
Given a regular ideal $J \subset A$ of height $d$ with 
the primary decomposition $J = \fm_1 \cap \cdots \cap \fm_r$,
we let $(J) = \stackrel{r}{\underset{i =1}\sum} \fm_i \in G^s_0(A)$.
Let $H^s_0(A) \subset G^s_0(A)$ be the subgroup generated by the classes 
$(J)$ as above such that $\mu(J) = d$. We let
$E^s_0(A) = {G^s_0(A)}/{H^s_0(A)}$.

We now consider the following commutative diagram of short exact sequences.
\begin{equation}\label{eqn:ECG-diagram}
\xymatrix@C1pc{
& 0 \ar[d] & 0 \ar[d] & 0 \ar[d] & \\
0 \ar[r] & F^s(A) \ar[r] \ar[d] & H^s(A) \ar[r] \ar[d] & H^s_0(A) \ar[r] 
\ar[d] & 0 \\ 
0 \ar[r] & T^s(A) \ar[r] \ar[d] & G^s(A) \ar[d] \ar[r] & G^s_0(A) \ar[r] 
\ar[d] & 0 \\
0 \ar[r] & L^s(A) \ar[r] \ar[d] & E^s(A) \ar[r]_{\psi^s_A} \ar[d] & E^s_0(A) 
\ar[d] \ar[r] & 0 \\
& 0 & 0 & 0 &}
\end{equation}
The only thing one needs to observe to get this diagram is that
the map $H^s(A) \to H^s_0(A)$ is surjective (by above definitions). 
It is easy to see that $T^s(A)$ is generated by
classes $(\fm, \omega_{\fm}) - (\fm, \omega'_{\fm})$, where
$\omega_{\fm}: (A/{\fm})^d \xrightarrow{\simeq} {\fm}/{\fm^2}$ and
$\omega'_{\fm}: (A/{\fm})^d \xrightarrow{\simeq} {\fm}/{\fm^2}$
are two isomorphisms.
It follows from ~\eqref{eqn:ECG-diagram} that the same holds for $L^s(A)$ as 
well. A similar commutative diagram
exists if we remove the superscript `$s$' everywhere.

\begin{lem}\label{lem:Ker-psi}
Let $k$ be an infinite field and let $A$ be a geometrically reduced affine $k$-algebra.  Then
$L^s(A) \subset E^s(A)$ is generated by elements of the type $(J, \omega_J)$,
where $J$ is a regular ideal of height $d$ with $\mu(J) = d$.
\end{lem}
\begin{proof}
Let $\wt{L}^s(A)$ denote the subgroup of $E^s(A)$ generated by
elements $(J, \omega_J)$, where $J$ is a regular ideal of height $d$ with 
$\mu(J) = d$. It is clear that $\wt{L}^s(A) \subseteq L^s(A)$. 
To prove the reverse inclusion, it suffices to show using the above
description of $L^s(A)$ that an element of the type 
$(\fm, \omega_{\fm}) - (\fm, \omega'_{\fm})$ lies in $\wt{L}^s(A)$.
The proof of this is a direct translation of \cite[Lemma~3.3]{BS-4}
and goes as follows.

If $(\fm, \omega'_{\fm}) = 0$ in $E^s(A)$, then it follows from 
\cite[Theorem~4.2]{BS-3} that $(\fm, \omega_{\fm}) \in \wt{L}^s(A)$.
So we can assume that $(\fm, \omega'_{\fm}) \neq 0$ in $E^s(A)$.
In this case, we can apply the Murthy-Swan Bertini theorem
(see the proof of \lemref{lem:Sm-ECG} below)
to find a regular ideal $I$ of height $d$ which is co-maximal with $\fm$
such that there is a surjection $\tau: A^d \surj J = \fm \cap I$ and
$\omega'_{\fm} = \tau|_{A/{\fm}}$.
If we let $\omega_{I} = \tau|_{A/{I}}$, then we get
$(\fm, \omega'_{\fm}) + (I, \omega_I) = (J, \tau|_{A/{J}}) = 0$ in
$E^s(A)$. 

On the other hand, since $J = \fm I$ and $\fm + I = A$, it follows that
$\omega_{\fm}$ and
$\omega_I$ induce a surjection $\omega_J: (A/{J})^d \surj {J}/{J^2}$
and hence $(\fm, \omega_{\fm}) + (I, \omega_I) = (J, \omega_J)$ in $E^s(A)$.
We conclude that $(\fm, \omega_{\fm}) - (\fm, \omega'_{\fm}) =
(J, \omega_J) - (J, \tau|_{A/{J}}) = (J, \omega_J) \in \wt{L}^s(A)$.
This proves the lemma.
\end{proof}

\subsection{Connection between the classical and modified Euler 
class groups}\label{sec:Connect}
We shall assume in this subsection that $k$ is an infinite field.
There is an obvious commutative diagram of the Euler class groups
\begin{equation}\label{eqn:Sm-weak-ECG-0} 
\xymatrix@C1pc{
E^s(A) \ar@{->>}[r]^-{\psi^s_A} \ar[d]_{\gamma_A} & E^s_0(A) \ar[d]^{\gamma^0_A} \\
E(A) \ar@{->>}[r]_-{\psi_A} & E_0(A).}
\end{equation}

The goal of this section is to prove that the vertical arrows are 
isomorphisms. We begin with the easy part of this goal.

\begin{lem}\label{lem:Sm-ECG}
Let $A$ be a geometrically reduced affine algebra of dimension $d \ge 2$ over $k$. 
Then there is a canonical isomorphism
\[
\gamma_A: E^s(A) \xrightarrow{\simeq} E(A).
\]
\end{lem}
\begin{proof}
We prove the surjectivity of $\gamma_A$ using the Bertini theorems of Murthy 
\cite[Theorem~2.3]{Murthy} and Swan \cite[Theorem~1.3]{Swan}.
Let $J$ be an ideal of $A$ of height $d$ with a surjection
$\omega_J: (A/J)^d \surj {J}/{J^2}$.
Let $\{\fm_1, \ldots, \fm_r\}$ be a set of smooth maximal ideals of $A$.
A special case of the Murthy-Swan Bertini theorem says that there exists an 
ideal $I \subset A$ (called a residual of $J$)
such that the following hold (see \cite[Corollary~2.6 and
Remarks~2.8, 3.2]{Murthy}). 
\begin{enumerate}
\item
There exists a surjection $\alpha: A^d \surj IJ$. 
\item
$I + J = I + \fm_i = A$ for $1 \le i \le r$.
\item
$I$ is a regular ideal of $A$ of height $d$.
\item
$\alpha|_{{A}/J} = \omega_J$.
\end{enumerate}

It follows from (1), (2) and (4) that $(I, \alpha|_{{A}/I}) + (J, \omega_J) =
(IJ, \alpha|_{A/{IJ}}) = 0$ in $E(A)$ and (3) says that 
$(I, \alpha|_{{A}/I}) \in E^s(A)$. This shows that $\gamma_A$ is surjective.

To show that $\gamma_A$ is injective, let $\alpha \in E^s(A)$ be such that
$\gamma_A(\alpha) = 0$. 
By repeatedly applying the above Bertini theorem again,
we can write $\alpha = (J, \omega_J)$, where $J$ is a regular ideal of height $d$ in $A$.
Indeed, for a general element $\alpha \in E^s(A)$, we choose a lift
$\wt{\alpha} = \sum a_{(\fm, \omega_{\fm})} (\fm, \omega_{\fm})$ in $G^s(A)$. We
let $m_{\wt{\alpha}}$ denote the number of $(\fm, \omega_{\fm})$ in the expression of $\wt{\alpha}$ such that
$a_{(\fm, \omega_{\fm})}$ is not equal to $\pm 1$. Using the above Bertini theorem and an induction on
$m_{\wt{\alpha}}$, we can assume that each $a_{(\fm, \omega_{\fm})} = \pm 1$. We can now apply the
Bertini theorem to those $(\fm, \omega_{\fm})$ such that $a_{(\fm, \omega_{\fm})}= -1$ and make all these
coefficients $1$. We can thus write $\wt{\alpha} = \sum (\fm, \omega_{\fm})$.
We now take $J$ to be the product of maximal ideals in this expression of $\wt{\alpha}$ and $\omega_J$ to be
the homomorphism induced from $\omega_{\fm}$'s.

We now apply \cite[Theorem~4.2]{BS-3} to conclude that
$\omega_J$ lifts to a surjection $\wt{\omega}_J: A^d \surj J$.
In particular, $\alpha = (J, \omega_J) = 0$ in $E^s(A)$. 
This shows that $\gamma_A$ is injective.
\end{proof}

Using \thmref{thm:Segre} and \lemref{lem:Sm-ECG}, we can now prove
our main comparison result.

\begin{prop}\label{prop:Sm-weak-ECG}
Let $A$ be a geometrically reduced affine algebra of dimension $d \ge 2$ over $k$. 
Then there is a canonical isomorphism
\[
\gamma^0_A: E^s_0(A) \xrightarrow{\simeq} E_0(A).
\]
\end{prop}
\begin{proof}
We have a commutative diagram of short exact sequences
\begin{equation}\label{eqn:Sm-weak-ECG-0-0} 
\xymatrix@C.8pc{
0 \ar[r] & L^s(A) \ar[r] \ar[d] & E^s(A) \ar[r] \ar[d]^{\simeq} & 
E^s_0(A) \ar[d] \ar[r]
& 0 \\
0 \ar[r] & L(A) \ar[r] & E(A) \ar[r] & E_0(A) \ar[r] & 0.}
\end{equation}
Using \lemref{lem:Sm-ECG}, it suffices to show that the left vertical
arrow in this diagram is surjective.
Using \thmref{thm:Segre}, it suffices to show that if 
$\underline{a} = [a_1, \ldots, a_{d+1}]$ is a unimodular row, then
$\phi_A([\underline{a}]) \in L^s(A)$. Note that we can identify
$L^s(A)$ with its image in $L(A)$.

Let $\{e_1, \ldots, e_d\}$ be the standard basis of the free $A$-module $A^d$.
Let $\alpha: A^d \to A$ be given by $\alpha(e_i) = a_i$ for $1 \le i \le d$.
For a projective $A$-module $P$ and an $A$-linear map $f: P^* \to A$,
let $Z(f)$ denote the closed subscheme of $\Spec(A)$ where 
the induced map $f^*: \Spec(A) \to \Spec({\rm Sym}(P^*))$ vanishes.

We fix a surjective $k$-algebra homomorphism $u:k[X_1, \ldots, X_n] \surj A$
and let $x_i = u(X_i)$ for $1 \le i \le n$.
Let $X = \Spec(A)$ and let $I \subset A$ be the reduced ideal such that 
$\Spec(A/I) = X_{\rm sing}$. Let $\ov{u}: k[X_1, \ldots, X_n] \surj A
\surj A/I$ be the composite map. For any $A$-module $M$ and an element
$m \in M$, let $\ov{m}$ denote its image under the map $M \surj M/IM$.
For $1 \le i \le n$ and $1 \le j \le d$, we let $t_{ij} = x_ie_j \in A^d$.

In this case, Swan's Bertini theorem (see the proof of 
\cite[Theorem~1.4]{Swan}) says that there exists a dense open subset
$U_1 \subset \A^{n(1+d)}_k$ such that for every $(\{\lambda_i\}, \{\gamma_{ij}\})
\in U_1(k)$, the following hold.
\begin{enumerate}
\item
$Z(\alpha + a_{d+1}(\sum_i \lambda_i e_i + \sum_{i,j} \gamma_{ij} t_{ij}))$
is a reduced closed subscheme of $X$ of pure codimension $\geq d$.
\item
$Z(\alpha + a_{d+1}(\sum_i \lambda_i e_i + \sum_{i,j} \gamma_{ij} t_{ij}))
\cap X_{\rm reg}$ is regular.
\end{enumerate}

Similarly, by applying the Bertini theorem to the composite
embedding $X_{\rm sing} \inj X \inj \A^n_k$, we get a dense open subset
$U_2 \subset \A^{n(1+d)}_k$ such that for every $(\{\lambda_i\}, \{\gamma_{ij}\})
\in U_2(k)$, the following holds. \\
\hspace*{.5cm} (3)
$Z(\ov{\alpha} + \ov{a_{d+1}}(\sum_i \lambda_i \ov{e_i} + \sum_{i,j} \gamma_{ij} 
\ov{t_{ij}}))$
is a  closed subscheme of $X_{\rm sing}$ of pure codimension\\
\hspace*{1.1cm} $d$ (or is empty).

Since $A$ is geometrically reduced and hence $\dim(X_{\rm sing}) \le d-1$, it follows from
(1), (2) and (3) that for a general set of elements
$\{b_1, \ldots, b_d\}$ in $A$, the ideal $J = (a_1 + b_1a_{d+1}, 
\ldots, a_d + b_d a_{d+1})$ has the property that either $J = A$ 
or it is a regular ideal of height $d$ in $A$. 

If $J =A$, we have $\phi_A([\underline{a}]) = 0$.
If $J$ is a regular ideal of height $d$, then 
it is easy to check that the  equivalence class of the unimodular row
$\underline{a'}:= [a'_1, \ldots, a'_d, a_{d+1}]$ in $WS_{d+1}(A)$
is same as that of $\underline{a}$, where $a'_i = a_i + b_ia_{d+1}$. Moreover, it follows from
\cite[\S~7, p.~214]{BS-3} when $d =2$ and from \cite[Remark~3.7]{DZ} when
$d \ge 3$ that $\phi_A([\underline{a}]) = (J, \ov{a_{d+1}}\omega_{J})$, where $\omega_J: (A/J)^d \to {J}/{J^2}$
is given by $\omega_J(e_i \ {\rm mod} \ J) = a'_i \ {\rm mod} \ J^2$.
Since $(J, \ov{a_{d+1}}\omega_{J}) \in G^s(A)$ and since $\mu(J) = d$,
it follows from Lemma~\ref{lem:Ker-psi} that $(J, \ov{a_{d+1}}\omega_{J}) 
\in L^s(A)$. This finishes the proof of the proposition.
\end{proof}

\subsection{Euler class group and $K$-theory}\label{sec:ECG-K}
Let $k$ be a field and let $A$ be an equi-dimensional geometrically
reduced affine algebra of dimension $d \ge 2$ over $k$.
One can check from the definition that a generator of $E_0(A)$ may
not be a local complete intersection ideal in $A$ in general.
So there is no evident cycle class map $E_0(A) \to K_0(A)$.
One immediate advantage of $E^s_0(A)$ is that each of its generator
is a local complete intersection ideal. Using this idea and  
\propref{prop:Sm-weak-ECG}, one can construct a cycle class map
$\cyc_A: E_0(A) \to K_0(A)$ as follows.

If $\fm \subset A$ is a regular maximal ideal, then $A/{\fm}$ admits
a class $[A/{\fm}] \in K_0(A)$. Extending it linearly, one gets
a map $G^s_0(A) \to K_0(A)$. If $J \subset A$ is a regular ideal of
height $d$ such that $\mu(J) = d$, then one knows that it must be
a complete intersection ideal (see \cite[Theorems~135, 125]{Kaplansky}).
Using the Koszul resolution of $A/J$, it easily follows that
$cyc_A((J)) = [A/J] = 0$ in $K_0(A)$. We therefore get a map
\begin{equation}\label{eqn:CCM-ECG}
cyc_A: E^s_0(A) \to K_0(A).
\end{equation}

Using \propref{prop:Sm-weak-ECG},
we can now prove our main result of this section:

\begin{thm}\label{thm:RR-ECG}
Let $k$ be an infinite field and 
let $A$ be a geometrically
reduced affine algebra of dimension $d \ge 2$ over $k$.
Assume that one of the following holds.
\begin{enumerate}
\item
$k$ is algebraically closed.
\item
$(d-1)! \in k^{\times}$.
\end{enumerate}

Then ${\rm Ker}(cyc_A)$ is a torsion group of exponent $(d-1)!$.
\end{thm}
\begin{proof}
Let $\alpha \in E^s_0(A)$ be such that $cyc_A(\alpha) = 0$. By
repeatedly applying the Murthy-Swan Bertini theorem, as in the
proof of \lemref{lem:Sm-ECG}, we can assume that $\alpha = (I)$, 
where $I \subset A$ is a regular ideal of height $d$. 
Our assumption then says that $[A/I] = 0$ in $K_0(A)$.
Since $I$ is supported on the Cohen-Macaulay locus of $A$, the proof
of \cite[Lemma~1.2]{Mandal} shows that there exists an 
$A$-regular sequence $(f_1, \dots, f_d)$ such that 
$I = (f_1, \dots, f_d) + I^2$. Let $J =(f_1, \dots, f_{d-1}) + I^{(d-1)!}$. 

It follows from \cite[Lemma~4.1]{Das-Mandal} that 
$(J) = (d-1)! (I)$ in $E_0(A)$. 
If we can show that $(J) = 0$ in $E_0(A)$, then it will follow that
$(d-1)!(I) = 0$ in $E_0(A)$. We can then conclude from 
\propref{prop:Sm-weak-ECG} that $(d-1)!(I) = 0$ in $E^s_0(A)$.
We have therefore reduced the problem to showing that
$(J) = 0 \in E_0(A)$.

Now, it follows from \cite[Theorem~2.2]{Murthy} that there exists a
projective $A$-module $P$ of rank $d$ and a surjection $P \surj J$ such that
$[P] - [A^d] = - [A/I]$ in $K_0(A)$. It follows from our hypothesis
on $I$ that $[P] = [A^d] \in K_0(A)$ so that $P$ is stably free.
If $k$ is algebraically closed, it follows from \cite[Theorem~6]{Suslin}
that $P$ is free. But this implies that $\mu(J) = d$ so that $(J) = 0$
in $E_0(A)$. 

Suppose now that $(d-1)! \in k^{\times}$. At any rate, it follows from
the cancellation theorem of Bass (see \cite[Theorem~1]{Suslin}) that
$P \oplus A \cong A^{d+1}$ so that $P$ is the kernel of a surjection
$A^{d+1} \surj A$. In this case, it is shown on \cite[Page~214]{BS-3}
that there exists an ideal $J'$ of height $d$ 
with $\mu(J') = d$ and a surjection $P \surj J'$.
Furthermore, under the assumption that $(d-1)! \in k^{\times}$,
it is shown in \cite[\S~4]{BS-3} that the weak Euler class $e_0(P)$ of $P$ 
is well defined in $E_0(A)$ and $(J) = e_0(P) = (J')$.
We are now done because $(J') = 0$ in $E_0(A)$.
This finishes the proof.
\end{proof}

Combining \propref{prop:Sm-weak-ECG} and \thmref{thm:RR-ECG}, we obtain
the following. When ${\rm char}(k) = 0$ and $A$ is Cohen-Macaulay, this
was earlier proven independently by Bhatwadekar (unpublished)
and Das-Mandal \cite[Corollary~4.2]{Das-Mandal}. 

\begin{cor}\label{cor:RR-ECG**}
Let $A$ be a geometrically
 reduced affine algebra of dimension $d \ge 2$ over 
an infinite field $k$.
Then there is a cycle class map $E_0(A) \to K_0(A)$ whose kernel
is torsion of exponent $(d-1)!$ if either $k = \ov{k}$ or
$(d-1)! \in k^{\times}$.
\end{cor}

\section{The kernel of the cycle class map}\label{sec:gen-field}
We shall now prove Theorems~\ref{thm:Thm-2}, ~\ref{thm:Thm-4}, 
~\ref{thm:Thm-6} and  ~\ref{thm:ECG-*} 
with the help of the results of \S~\ref{sec:ECG}.
In order to do so, we need to recall the cycle class map
for 0-cycles in ~\eqref{eqn:CCM-*-3} in the modulus setting.

\subsection{The cycle class map with modulus}\label{sec:CCM-M}
Let $X$ be a smooth quasi-projective scheme of dimension $d\geq 1$ over a perfect field $k$
and let $D \subset X$ be an effective Cartier divisor. Recall from
\cite[Theorem~12.4]{BK} that there is a cycle class map with modulus 
\begin{equation}\label{eqn:CCM-BK}
\lambda_{X|D} : \CH_0(X| D) \to K_0(X,D).
\end{equation}

This was constructed
as the composition of the left arrows in the following commutative diagram of 
short exact
 sequences.
\begin{equation}\label{eqn:CCM-BK-1}
\xymatrix@C1pc{
0 \ar[r]&  \CH_0(X| D) \ar[r]^-{{p_{+}}_*} \ar[d]_-{\tilde{\lambda}_{X|D}} & 
\CH_0(S_X) \ar[d]^{\lambda_{S_X}} \ar[r]^{\iota^*_-} & \CH_0(X) \ar[r] 
\ar[d]^{\lambda_X}& 0   \\
0 \ar[r]&  K_0(S_X, X_-) \ar[r]^-{{p_{+}}_*}  \ar[d]_-{\phi_0}& K_0(S_X) 
\ar[r]^-{\iota^*_-} \ar[d]^{\iota^*_+}& K_0(X) \ar[d]^{\iota^*} \ar[r]& 0   \\
 &  K_0(X, D) \ar[r] & K_0(X) \ar[r]^-{\iota^*} & K_0(D),&   \\
}
\end{equation}
where $\lambda_{S_X}$ and $\lambda_X$ are as in ~\eqref{eqn:CCM-*-3}.

Observe that, given $x\in z_0(X|D)$, the composition of the maps of spectra 
$K(k(x)) \to K(X) \to K(D)$ 
is null homotopic and hence it defines a map $K_0(k(x)) \to K_0(X, D)$. 
Extending linearly, we have a homomorphism $z_0(X|D) \to K_0(X, D)$. 
Moreover, this homomorphism factors through $\phi_0: K_0(S_X, X_+) \to 
K_0(X, D)$ and,
on cycles, it is the same as $\lambda_{X|D}$. To see this, let 
$p_{+,x} : \Spec(k(x)) \hookrightarrow S_{X}$ be the 
composition of the inclusions $\iota_x: \Spec(k(x)) \hookrightarrow X$ and 
$\iota_{+} : X \hookrightarrow S_X$. 
We then have the commutative diagram
\begin{equation}\label{eqn:CCM-BK-2}
\xymatrix@C1pc{
K(k(x)) \ar[r]  \ar@/^1.5pc/[rr]^-{{p_{+,x}}_*} \ar@{=}[d] & 
K^{\{x\}}(X_+ \setminus D) \ar[d]^{\iota^*_+}_{\cong}&K^{\{x\}}(S_X) \ar[r]  
\ar[l]_-{\cong} \ar[d]^{\iota^*_+} & K(S_X, X_-)\ar[d]^{\phi_0} \\
K(k(x)) \ar[r]  \ar@/_1.5pc/[rr]_-{{\iota_{x}}_*} & K^{\{x\}}(X \setminus D) &
K^{\{x\}}(X) \ar[r]  \ar[l]_-{\cong}& K(X, D)
 }
 \end{equation}
such that the composition of the arrows in the top and bottom rows give the 
maps $\tilde{\lambda}_{X|D}$ and $\lambda_{X|D}$, respectively.

\subsection{Proof of  Theorem~\ref{thm:Thm-6}}\label{sec:kernel-cycle}
We fix a field $k$. We also fix a (equi-dimensional) geometrically
reduced affine algebra $A$ of dimension $d \ge 2$ over $k$ and let 
$X = \Spec(A)$. We shall interchangeably use the notations $\CH^{LW}_0(X)$
and $\CH^{LW}_0(A)$. We begin with the following connection between the
Euler class group and the Chow group of $A$.

It is clear from the definition of $E^s_0(A)$ in \S~\ref{sec:MECG} that  
there are canonical maps $G^s_0(A) \xrightarrow{\cong} \sZ_0(A)
\surj \CH^{LW}_0(A)$
which sends a regular maximal ideal $\fm$ to the cycle class 
$[x] \in \sZ_0(A)$, where $x = \Spec({A}/{\fm}) \in X_{\rm reg}$.
In order to show that the composite map factors through $E^s_0(A)$, 
we need to show that if $J \subset A$ is a regular ideal of height $d$ such 
that $\mu(J) = d$, then the image of $(J)$ in $\sZ_0(A)$ lies in
$\sR^{LW}_0(A)$. By \cite[Lemma~2.2]{LW}, it suffices to show that
$J$ is a complete intersection ideal. But this follows directly from
\cite[Theorems~125, 135]{Kaplansky} because a regular ideal is always
a local complete intersection. We have therefore constructed
a canonical surjective map
\begin{equation}\label{eqn:Euler-Chow}
\wt{cyc}_A: E^s_0(A) \surj \CH^{LW}_0(A)
\end{equation}
and it is immediate from ~\eqref{eqn:CCM-ECG} that there is a 
commutative diagram
\begin{equation}\label{eqn:Euler-Chow-0}
\xymatrix@C.8pc{
E_0(A) & E^s_0(A) \ar@{->>}[r]^-{\wt{cyc}_A} \ar@/_2pc/[rrr]_{cyc_A} 
\ar[l]_{\gamma^0_A}
&  \CH^{LW}_0(A) \ar@{->>}[r] \ar@/^2pc/[rr]^{\lambda_A} & \CH_0(A) \ar[r] & 
K_0(A).}
\end{equation}

As a combination of \thmref{thm:RR-ECG} and ~\eqref{eqn:Euler-Chow-0}, we
immediately get the following result about the cycle class map for the
0-cycles. When $k = \ov{k}$, this gives an independent proof of
an old unpublished result of Levine (see \cite[Corollary~5.4]{Levine-5}). 
When $k$ is not algebraically closed, this result is completely new.

\begin{thm}\label{thm:RR-0-cycle}
Let $A$ be geometrically
reduced affine algebra of dimension $d \ge 2$ over an 
field $k$. Assume that either $k$ is algebraically
closed or $(d-1)! \in k^{\times}$. Let $X = \Spec(A)$ and let $D \subset X$ be an
effective Cartier divisor. Then the following hold.  
\begin{enumerate}
\item If $k$ is an infinite field, then the kernel of the cycle class
map $\lambda_A:\CH^{LW}_0(A) \to K_0(A)$ is a torsion group of exponent
$(d-1)!$.
\item 
For arbitrary fields, the kernel of the cycle class
map $\lambda_A:\CH_0(A) \to K_0(A)$ is a torsion group of exponent
$(d-1)!$.
\item If $k$ is perfect and $X$ is smooth, then the kernel of the  cycle class map 
$\lambda_{X|D}: \CH_0(X|D) \to K_0(X,D)$ as in ~\eqref{eqn:CCM-BK}
is a torsion group of exponent
$(d-1)!$.
\end{enumerate}
\end{thm} 

\begin{proof}
When $k$ is infinite, \thmref{thm:RR-ECG} and ~\eqref{eqn:Euler-Chow-0} together prove (1) and (2). 
We now assume $k$ is finite and prove (2).
We denote the map $\CH_0(A) \to K_0(A)$ by $\wt{\lambda}_A$.
Let $\alpha \in \CH_0(A)$ be such that $\wt{\lambda}_A(\alpha) = 0$. Let $\beta =(d-1)! \alpha \in \CH_0(A)$.
We choose two distinct primes $\ell_1$ and $\ell_2$ different from 
${\rm char}(k)$ and let $k_i$ denote the pro-$\ell_i$ extension of $k$ for 
$i = 1,2$.

It follows from the case of infinite fields, the compatibility of
the cycle class map with respect to field extensions and
\cite[Proposition~6.1]{BK} that $\beta_{k_i} = 0$ for $i =1,2$.
Note that each $k_i$ is a limit of finite separable extensions of the perfect
field $k$ and hence the hypotheses of \cite[Proposition~6.1]{BK} are satisfied.
Another application of \cite[Proposition~6.1]{BK} shows that
we can find two finite extensions $k_1'$ and $k_2'$ of $k$ of relatively prime 
degrees such that $\beta_{k'_i\D5} = 0$ for $i =1,2$. 
We conclude by applying \cite[Proposition~6.1]{BK} once again 
that $(d-1)! \alpha =  \beta = 0$.

Now 
assume that $k$ is perfect and 
$X=\Spec(A)$ is smooth. 
Then by (2) it follows that the kernel of the cycle 
class map 
$\lambda_{S_X} : \CH_0(S_X) \to K_0(S_X)$ is of exponent $(d-1)!$. It 
follows from ~\eqref{eqn:CCM-*-3} and ~\eqref{eqn:CCM-BK-1} that the
same is true of the kernel of the map $\tilde{\lambda}_{X|D} : \CH_0(X|D) \to 
K_0(S_X, X_+)$. 
The assertion (3) then follows from \cite[Lemma ~4.1]{Milnor} which yields 
that the natural 
map $\phi_0: K_0(S_X, X_-) \to K_0(X, D)$ is an isomorphism.
\end{proof}

\subsection{Bloch's formula for affine surfaces over arbitrary field}
\label{sec:BQ-arb}
As a corollary, we can prove Theorems~\ref{thm:Thm-2}, ~\ref{thm:Thm-4}
and ~\ref{thm:ECG-*} as follows. The part (2) of the theorem below was
conjectured by Bhatwadekar and R. Sridharan 
(see \cite[Remark~3.13]{BS-4}). 
Assuming that $A$ is regular, this conjecture was proven 
by Bhatwadekar (unpublished) in dimension two and,
by Asok and Fasel \cite{AF} in general. 

\begin{thm}\label{thm:Final-surface}
Let $k$ be a field and let $A$ be a geometrically
 reduced affine algebra of
dimension two over $k$. Let $X = \Spec(A)$ and let $D \subset X$ be an
effective Cartier divisor. Then the following hold.
\begin{enumerate}
\item
The map $\CH_0(X) \to K_0(X)$ is injective.
\item
If $k$ is infinite, then there are isomorphisms
$E_0(A) \xrightarrow{\cong} \CH^{LW}_0(A) \xrightarrow{\cong} \CH_0(A)$.
\item If $k$ is a perfect field, then 
the map $\rho_X: \sZ_0(X) \to H^2_{\nis}(X, \sK^M_{2,X})$ induces an
isomorphism $\rho_X: \CH_0(X) \xrightarrow{\cong} H^2_{\nis}(X, \sK^M_{2,X})$.
\item
If $k$ is a perfect field and $X$ is smooth, there is an isomorphism 
$\rho_{X|D}: \CH_0(X|D) \xrightarrow{\cong} H^2_{\nis}(X, \sK^M_{2,(X,D)})$.
\end{enumerate}
\end{thm}
\begin{proof}

When $k$ is infinite, it follows from \propref{prop:Sm-weak-ECG},
\thmref{thm:RR-ECG} and ~\eqref{eqn:Euler-Chow-0} that  the maps
\[
E_0(A) \xleftarrow{\gamma^0_A} 
E^s_0(A) \xrightarrow{\wt{cyc}_A} \CH^{LW}_0(A) \surj \CH_0(A)
\]
are all isomorphisms and the map $\CH_0(A) \to K_0(A)$ is injective.
This proves (2) while the assertion (1) follows from \thmref{thm:RR-0-cycle}.

We now prove (3). By \cite[Proposition 14]{kerz10}, the natural map of sheaves 
$\sK^M_{2,X} \to \sK_{2,X}$ is surjective. Since they are generically same, it 
follows that $H^2_{\nis}(X, \sK^M_{2,X}) \to H^2_{\nis}(X, \sK_{2,X})$ is an 
isomorphism. 
The existence of $\rho_X: \CH_0(X) \to 
H^2_{\nis}(X, \sK^M_{2,X})$ follows immediately from \lemref{lem:CCM-K*}
and \cite[Lemma~2.1]{Krishna-1}, which shows that the map
$\kappa_X: H^2_{\zar}(X, \sK_{2,X}) \xrightarrow{\cong}
H^2_{\nis}(X, \sK_{2,X}) \to K_0(X)$ is injective.
It follows from (1) and \lemref{lem:CCM-K*}
that $\rho_X$ is injective and its surjectivity 
follows from \cite[Theorem~2.5]{Kato-Saito} (see the proof of
\thmref{thm:Killing}). 
This proves (3). The assertion (4) follows from (3) and 
\thmref{thm:BS-main} exactly as we proved  \thmref{thm:Thm-3} 
in \S~\ref{sec:Thm-3-prf}. 
\end{proof}

\begin{remk}\label{remk:Fin-proj}
We warn the reader that the affineness of $X$ 
is an essential condition in \thmref{thm:Final-surface} (3)-(4).
One should not expect 
Bloch's formula $\CH_0(X|D) \xrightarrow{\cong} H^2_{\nis}(X, \sK^M_{2,(X,D)})$
when $X$ is a smooth projective surface over a finite field. 
The reason for this is that a result of Kerz and Saito \cite{KeS}
says that there is an isomorphism 
$\CH_0(X|D)_0 \xrightarrow{\cong} \pi_1(X,D)_0$, where
$\pi_1(X,D)$ is a quotient of the {\'e}tale fundamental group of 
$\pi_1(X \setminus D)$ which characterizes abelian {\'e}tale covers
of $X \setminus D$ which have ramification along $|D|$ bounded by the
divisor $D$.

On the other hand, a result of Kato and Saito \cite{Kato-Saito} implies that
there is a surjection $\pi_1(X,D) \surj  H^2_{\nis}(X, \sK^M_{2,(X,D)})$
which is not expected to be an isomorphism in general.
This suggests that the relative Milnor $K$-theory needs to be suitably
re-defined in order to solve this anomaly.
\end{remk}

\vskip .4cm

\noindent\emph{Acknowledgments.} 
The authors would like to thank the referees for carefully reading
the paper and providing many valuable suggestions which considerably improved its presentation.

\end{document}